\pdfoutput=1
\RequirePackage{ifpdf}
\ifpdf
\documentclass[pdftex]{sigma}
\else
\documentclass{sigma}
\fi

\usepackage{bm,enumitem,mathrsfs,mathtools,mleftright,soul}

\numberwithin{equation}{section}

\newtheorem{Thm}{Theorem}[section]
\newtheorem{Lem}[Thm]{Lemma}
\newtheorem{Prop}[Thm]{Proposition}

\theoremstyle{definition}
\newtheorem{Eg}[Thm]{Example}
\newtheorem{Def}[Thm]{Definition}
\newtheorem{Rmk}[Thm]{Remark}

\DeclareFontFamily{U}{matha}{\hyphenchar\font45}
\DeclareFontShape{U}{matha}{m}{n}{
 <5> <6> <7> <8> <9> <10> gen * matha
 <10.95> matha10 <12> <14.4> <17.28> <20.74> <24.88> matha12
 }{}
\DeclareSymbolFont{matha}{U}{matha}{m}{n}
\DeclareFontSubstitution{U}{matha}{m}{n}

\DeclareFontFamily{U}{mathx}{\hyphenchar\font45}
\DeclareFontShape{U}{mathx}{m}{n}{
 <5> <6> <7> <8> <9> <10>
 <10.95> <12> <14.4> <17.28> <20.74> <24.88>
 mathx10
 }{}
\DeclareSymbolFont{mathx}{U}{mathx}{m}{n}
\DeclareFontSubstitution{U}{mathx}{m}{n}

\DeclareMathDelimiter{\VERT}{0}{matha}{"7E}{mathx}{"17}

\newcommand{\C}{\mathbb{C}}
\newcommand{\D}{\mathsf{D}}
\newcommand{\e}[1]{\mathbf{e}_{#1}}
\newcommand{\I}{\mathsf{I}}
\newcommand{\M}{\mathsf{M}}
\newcommand{\N}{\mathbb{N}}
\newcommand{\R}{\mathbb{R}}
\newcommand{\T}{\mathbb{T}}
\newcommand{\X}{\mathsf{X}}
\newcommand{\Z}{\mathbb{Z}}
\newcommand{\ad}[1]{\operatorname{ad}_{#1}}
\newcommand{\bn}{\mathsf{bn}}
\newcommand{\Br}[1]{\( #1 \)}
\newcommand{\CC}[2]{\SqBr{#1,#2}}
\newcommand{\Cl}[2]{\overline{#1}^{#2}}
\newcommand{\CO}[2]{\[ #1,#2 \)}
\newcommand{\df}{\stackrel{\operatorname{df}}{=}}
\newcommand{\DM}[1]{\D_{#1,\M}}
\newcommand{\dn}{\mathsf{dn}}
\newcommand{\Id}{\operatorname{Id}}
\newcommand{\mk}{\mathsf{mk}}
\newcommand{\sa}[1]{#1_{\operatorname{sa}}}
\newcommand{\abs}[1]{\mleft| #1 \mright|}
\newcommand{\Adj}[2]{\func{\mathbb{L}}{#1,#2}}
\newcommand{\CcS}[1]{\func{C_{c}^{\infty}}{#1}}
\newcommand{\Der}[1]{\mathcal{D}_{#1}}
\newcommand{\Dom}[1]{\func{\operatorname{Dom}}{#1}}
\newcommand{\Inn}{\Inner{\cdot}{\cdot}}
\newcommand{\Int}[4]{\int_{#1} #2 \, \mathrm{d}{\func{#3}{#4}}}
\newcommand{\Seq}[2]{\Br{#1}_{#2}}
\newcommand{\Set}[2]{\mleft\{ \left. #1 \right| #2 \mright\}}
\newcommand{\ska}[1]{#1_{\operatorname{ska}}}
\newcommand{\DArg}[2]{\func{\D_{#1}}{#2}}
\newcommand{\func}[2]{#1 \Br{#2}}
\newcommand{\Haus}[3]{\func{\mathsf{Haus}_{#1}}{#2,#3}}
\newcommand{\Norm}[1]{\mleft\| #1 \mright\|}
\newcommand{\Pair}[2]{\Br{#1,#2}}
\newcommand{\Quot}[2]{#1 / #2}
\newcommand{\SqBr}[1]{\[ #1 \]}
\newcommand{\SSet}[1]{\mleft\{ #1 \mright\}}
\newcommand{\Trip}[3]{\Br{#1,#2,#3}}
\newcommand{\bnArg}[3]{\func{\bn_{#1}}{#2,#3}}
\newcommand{\CoDom}[1]{\func{\operatorname{Co-Dom}}{#1}}
\newcommand{\DMArg}[2]{\func{\DM{#1}}{#2}}
\newcommand{\dnArg}[3]{\func{\dn_{#1}}{#2,#3}}

\newcommand{\Inner}[2]{\left\langle \left. #1 \right| #2 \right\rangle}
\newcommand{\InnerR}[2]{\left\langle #1 \left| #2 \right. \right\rangle}
\newcommand{\Range}[1]{\func{\operatorname{Range}}{#1}}
\newcommand{\State}[1]{\func{\mathscr{S}}{#1}}
\newcommand{\OpNorm}[1]{\mleft\VERT #1 \mright\VERT}
\newcommand{\Smooth}[1]{#1^{\operatorname{\infty}}}
\newcommand{\ModProp}[5]{\func{\Lambda^{\mod}_{#1,#2,#3}}{#4,#5}}

\renewcommand{\(}{\mleft(}
\renewcommand{\)}{\mright)}
\renewcommand{\[}{\mleft[}
\renewcommand{\]}{\mright]}
\renewcommand{\d}{\displaystyle}
\renewcommand{\k}{\mathsf{k}}
\renewcommand{\L}{\mathsf{L}}
\renewcommand{\st}{\operatorname{st}}
\renewcommand{\ker}[1]{\func{\operatorname{ker}}{#1}}
\renewcommand{\mod}{\mathsf{mod}}
\renewcommand{\alph}[2]{\func{\alpha_{#1}}{#2}}

\begin{document}

\newcommand{\arXivNumber}{1803.04036}

\renewcommand{\PaperNumber}{079}

\FirstPageHeading

\ShortArticleName{Metrized Quantum Vector Bundles over Quantum Tori}
\ArticleName{Metrized Quantum Vector Bundles\\ over Quantum Tori Built from Riemannian Metrics\\ and Rosenberg's Levi-Civita Connections}

\Author{Leonard HUANG}
\AuthorNameForHeading{L.~Huang}

\Address{Department of Mathematics, University of Colorado at Boulder,\\ Campus Box 395, 2300 Colorado Avenue, Boulder, CO 80309-0395, USA}

\Email{\href{mailto:Leonard.Huang@Colorado.EDU}{Leonard.Huang@Colorado.EDU}}

\ArticleDates{Received March 13, 2018, in final form July 21, 2018; Published online July 29, 2018}

\Abstract{We build metrized quantum vector bundles, over a generically transcendental quantum torus, from Riemannian metrics, using Rosenberg's Levi-Civita connections for these metrics. We also prove that two metrized quantum vector bundles, corresponding to positive scalar multiples of a Riemannian metric, have distance zero between them with respect to the modular Gromov--Hausdorff propinquity.}


\Classification{46L08; 46L57; 46L87; 37A55; 58B34}

\section{Introduction}

This paper was inspired by an apparent connection between Jonathan Rosenberg's work on Riemannian metrics on a generically transcendental quantum torus and Levi-Civita connections for these metrics~\cite{Ro}, and Fr\'ed\'eric Latr\'emoli\`ere's work on metrized quantum vector bundles and the modular Gromov--Hausdorff propinquity~\cite{L4}.

The subject of the Gromov--Hausdorff propinquity has its origin in Marc Rieffel's observation~\cite{Ri3} that in certain papers on theoretical physics, statements can be found regarding the convergence of a sequence of operator algebras to an operator algebra. He deduced that the bookkeeping device used by the authors of these papers to prove convergence can be described as a metric structure on unital $ C^{\ast} $-algebras. Seeing that the Gromov--Hausdorff distance for compact metric spaces enables us to discuss the convergence of a sequence of compact metric spaces to a compact metric space, he defined for the class of order-unit spaces endowed with a special metric structure (called the \emph{compact quantum metric spaces}) an analogous distance called the \emph{quantum Gromov--Hausdorff distance}. The relation between order-unit spaces and $ C^{\ast} $-algebras is made clear when one knows that the space of self-adjoint elements of a unital $ C^{\ast} $-algebra is an order-unit space.

The quantum Gromov--Hausdorff distance suffers from some deficiencies. Designed for order-unit spaces, it does not incorporate the multiplicative structure of a $ C^{\ast} $-algebra. Also, it was unknown if distance zero between two $ C^{\ast} $-algebras necessarily means that they are $ \ast $-isomorphic. These problems were settled when Latr\'emoli\`ere defined, in~\cite{L3}, the \emph{quantum Gromov--Hausdorff propinquity} for the class of unital $ C^{\ast} $-algebras endowed with a special metric structure (called the \emph{quantum compact metric spaces}).

Recently, Latr\'emoli\`ere was able to generalize the Gromov--Hausdorff propinquity to Hilbert $ C^{\ast} $-modules over a quantum compact metric space endowed with a special metric structure~\cite{L4}. He calls these objects \emph{metrized quantum vector bundles}, regarding them to be a noncommutative generalization of vector bundles over a~compact Riemannian manifold endowed with a~metric. The \emph{modular Gromov--Hausdorff propinquity} then allows us to formalize the concept of convergence for metrized quantum vector bundles. Distance zero between metrized quantum vector bundles is equivalent to the existence of an isomorphism between them, in terms of their Hilbert-$ C^{\ast} $-module structures and their metric structures.

In Section~\ref{section3}, we will define Riemannian metrics on generically transcendental quantum tori and provide a brief overview of Rosenberg's work on Levi-Civita connections for these metrics.

In Section~\ref{section4}, we will define Latr\'emoli\`ere's quantum compact metric spaces and metrized quantum vector bundles. We will then show how to build a metrized quantum vector bundle from a Riemannian metric on a generically transcendental quantum torus and the Levi-Civita connection for the metric.

In Section~\ref{section5}, we will prove that two metrized quantum vector bundles, corresponding to positive scalar multiples of a Riemannian metric, have distance zero between them with respect to the modular Gromov--Hausdorff propinquity.

In Section~\ref{section6}, we will pose some open questions that serve as the basis for future work in this area.

\section{Preliminaries}\label{section2}
This section serves to standardize notation and conventions.

Let $ \N $ denote the set of positive integers, and for each $ m \in \N $, let $ \SqBr{m} \df \N_{\leq m} $.

Throughout this paper, fix $ n \in \N $ as well as a generically transcendental\footnote{This is a rather complicated Diophantine condition that is defined in~\cite{BEJ}.} skew-adjoint $ \Br{n \times n} $-matrix $ \Theta $ having entries in $ \C $.

Fix also an arbitrary norm $ N $ on $ \R^{n} $. Of particular physical importance is the Euclidean norm on $ \R^{n} $.

Let $ A_{\Theta} $ denote the $ n $-dimensional quantum torus corresponding to $ \Theta $, which is the universal $ C^{\ast} $-algebra generated by $ n $ unitary elements $ u_{1},\ldots,u_{n} $ satisfying the relation $ u_{k} u_{j} = e^{2 \pi i \Theta_{jk}} u_{j} u_{k} $ for all $ j,k \in \SqBr{n} $.

Let $ \alpha $ denote the canonical action of $ \T^{n} $ on $ A_{\Theta} $ given by $ \alph{\bm{t}}{u_{j}} = t_{j} \cdot u_{j} $ for all $ \bm{t} \in \T^{n} $ and $ j \in \SqBr{n} $, where $ \bm{t} = \Seq{t_{j}}{j \in \SqBr{n}} $. Then let $ \partial_{1},\ldots,\partial_{n} $ denote the coordinate directional-derivative operators associated to $ \alpha $.

Let $ \Smooth{A_{\Theta}} $ denote the $ \ast $-subalgebra of smooth elements of $ A_{\Theta} $ for $ \alpha $, which we refer to as the $ n $-dimensional \emph{smooth quantum torus} corresponding to $ \Theta $. Then let $ \Der{\Theta} $ denote the $ \R $-vector space of $ \ast $-derivations on $ \Smooth{A_{\Theta}} $.

Given a vector space $ V $, a seminorm $ L $ on $ V $, and $ r \in \R_{> 0} $, let $ L^{r} \df \Set{v \in V}{\func{L}{v} \leq r} $.

Given a $ \ast $-algebra $ A $, let its set of self-adjoint elements and its set of skew-adjoint elements be denoted by $ \sa{A} $ and $ \ska{A} $ respectively, and the $ \ast $-algebra of $ \Br{n \times n} $-matrices with entries in~$ A $ be denoted by~$ \func{M_{n}}{A} $.

Given a $ C^{\ast} $-algebra $ A $, let its state space be denoted by $ \State{A} $.

\section[Generically transcendental quantum tori as noncommutative Riemannian manifolds]{Generically transcendental quantum tori\\ as noncommutative Riemannian manifolds}\label{section3}

This section gives a brief overview of Riemannian metrics on a generically transcendental quantum torus and Rosenberg's Levi-Civita connections for these metrics, as defined in \cite{Ro}.

Throughout this section, $ A $ denotes a unital $ C^{\ast} $-algebra.

\begin{Def}Let $ \func{\chi}{A} $ denote the free left $ A $-module with rank $ n $ whose underlying $ \C $-vector space is $ A^{n} $, where the left action $ \bullet $ of $ a \in A $ on $ X \in \func{\chi}{A} $ is given by left multiplication of each component of $ X $ by $ a $. For every $ j \in \SqBr{n} $, let $ \e{j} $ denote the element of $ \func{\chi}{A} $ that has $ 1_{A} $ in the $ j $-th component and $ 0_{A} $'s elsewhere. Define the standard $ A $-valued inner product $ \Inn_{\st} $ and its associated (standard) norm $ \Norm{\cdot}_{\st} $ on $ \func{\chi}{A} $ by
\begin{align*}
\forall\, a_{1},\ldots,a_{n},b_{1},\ldots,b_{n} \in A\colon \quad & \Inner{\sum_{j = 1}^{n} a_{j} \bullet \e{j}}{\sum_{j = 1}^{n} b_{j} \bullet \e{j}}_{\st} \df \sum_{j = 1}^{n} a_{j} b_{j}^{\ast}, \\
& \Norm{\sum_{j = 1}^{n} a_{j} \bullet \e{j}}_{\st} \df \Norm{\sum_{j = 1}^{n} a_{j} a_{j}^{\ast}}_{A}^{\frac{1}{2}}.
\end{align*}
Then $ \Pair{\func{\chi}{A}}{\Inn_{\st}} $ is a left Hilbert $ A $-module.

Let $ \Adj{\func{\chi}{A}}{\Inn_{\st}} $ denote the $ C^{\ast} $-algebra of adjointable maps on $ \Pair{\func{\chi}{A}}{\Inn_{\st}} $. Define a~unital algebraic $ \ast $-anti-isomorphism $ T\colon \func{M_{n}}{A} \to \Adj{\func{\chi}{A}}{\Inn_{\st}} $ by
\begin{gather*}
\forall\, a_{1},\ldots,a_{n} \in A\colon \quad \func{T_{g}}{\sum_{j = 1}^{n} a_{j} \bullet \e{j}} \df \sum_{k = 1}^{n} \Br{\sum_{j = 1}^{n} a_{j} g_{j k}} \bullet \e{k}.
\end{gather*}
We can define a $ C^{\ast} $-algebraic norm $ \Norm{\cdot}_{\func{M_{n}}{A}} $ on $ \func{M_{n}}{A} $ by $ \Norm{g}_{\func{M_{n}}{A}} \df \Norm{T_{g}}_{\Adj{\func{\chi}{A}}{\Inn_{\st}}} $ for all $ g \in \func{M_{n}}{A} $. Henceforth, we will view $ \func{M_{n}}{A} $ as a $ C^{\ast} $-algebra.
\end{Def}

\begin{Prop} \label{Properties of a Riemannian Metric} Let $ g \!\in\! \func{M_{n}}{A} $ be positive and invertible. Define a map \smash{$ \Inn_{g}\!\colon \! \func{\chi}{A} \!\times\! \func{\chi}{A} \!\to\! A\!$} by
\begin{gather*}
\forall\, X,Y \in \func{\chi}{A}\colon \quad \Inner{X}{Y}_{g} \df \Inner{\func{T_{g}}{X}}{Y}_{\st}.
\end{gather*}
Then the following statements hold:
\begin{enumerate}\itemsep=0pt
\item[$1)$] $ \big(\func{\chi}{A},\Inn_{g}\big) $ is a left Hilbert $ A $-module $($we denote the associated norm on $ \func{\chi}{A} $ by $ \Norm{\cdot}_{g})$,
\item[$2)$] $ \Norm{X}_{g} = \big\|\func{T_{\sqrt{g}}}{X}\big\|_{\st} $ for all $ X \in \func{\chi}{A} $,
\item[$3)$] $ \Inner{\e{j}}{\e{k}}_{g} = g_{j k} $ for all $ j,k \in \SqBr{n} $.
\end{enumerate}
\end{Prop}

\begin{proof}By hypothesis, $ g $ has an invertible positive square root in $ \func{M_{n}}{A} $, so
\begin{gather*}
\forall\, X,Y \in \func{\chi}{A}\colon \quad
 \Inner{X}{Y}_{g}
= \Inner{\func{T_{g}}{X}}{Y}_{\st}
= \big\langle{\func{T_{\sqrt{g}}}{X}}\,|\,{\func{T_{\sqrt{g}}^{\ast}}{Y}}\big\rangle_{\st}
= \big\langle{\func{T_{\sqrt{g}}}{X}}\,|\,{\func{T_{\sqrt{g}}}{Y}}\big\rangle_{\st}.
\end{gather*}
We can thus see that $ \Inn_{g} $ is a sesquilinear form. Furthermore,
\begin{gather*}
\forall\, X,Y \in \func{\chi}{A}\colon \quad
 \Inner{Y}{X}_{g}
= \big\langle {\func{T_{\sqrt{g}}}{Y}}\,|\,{\func{T_{\sqrt{g}}}{X}}\big\rangle_{\st}
= \big\langle{\func{T_{\sqrt{g}}}{X}}\,|\,{\func{T_{\sqrt{g}}}{Y}}\big\rangle_{\st}^{\ast}
= \Inner{X}{Y}_{g}^{\ast}.
\end{gather*}
As $ T_{\sqrt{g}} $ is $ A $-linear, we find that
\begin{align*}
\forall\, a \in A, ~ \forall\, X,Y \in \func{\chi}{A}\colon \quad
 \Inner{a \bullet X}{Y}_{g}
& = \big\langle{\func{T_{\sqrt{g}}}{a \bullet X}}\,|\,{\func{T_{\sqrt{g}}}{Y}}\big\rangle_{\st} \\
& = \big\langle{a \bullet \func{T_{\sqrt{g}}}{X}}\,|\,{\func{T_{\sqrt{g}}}{Y}}\big\rangle_{\st} \\
& = a \big\langle {\func{T_{\sqrt{g}}}{X}}\,|\,{\func{T_{\sqrt{g}}}{Y}}\big\rangle_{\st} \\
& = a \Inner{X}{Y}_{g}.
\end{align*}
Next, we have for all $ X \in \func{\chi}{A} $ that
\begin{gather*}
 \Inner{X}{X}_{g}= \big\langle{\func{T_{\sqrt{g}}}{X}}\,|\,{\func{T_{\sqrt{g}}}{X}}\big\rangle_{\st}
\geq_{A} 0_{A},
\end{gather*}
and if $ \Inner{X}{X}_{g} = 0_{A} $, then $ \big\langle{\func{T_{\sqrt{g}}}{X}}\,|\,{\func{T_{\sqrt{g}}}{X}}\big\rangle_{\st} = 0_{A} $, so $ \func{T_{\sqrt{g}}}{X} = 0_{\func{\chi}{A}} $ and thus $ X = 0_{\func{\chi}{A}} $. Continuing,
\begin{gather*}
\forall\, X \in \func{\chi}{A}\colon \quad \Norm{X}_{g}
= \big\|\Inner{X}{X}_{g}\big\|_{A}^{\frac{1}{2}}
= \big\|\big\langle{\func{T_{\sqrt{g}}}{X}}\,|\,{\func{T_{\sqrt{g}}}{X}}\big\rangle_{\st}\big|_{A}^{\frac{1}{2}}
= \big\|\func{T_{\sqrt{g}}}{X}\big\|_{\st},
\end{gather*}
so if $ \Seq{X_{k}}{k \in \N} $ is a Cauchy sequence in $ \func{\chi}{A} $ with respect to $ \Norm{\cdot}_{g} $, then $ \big({\func{T_{\sqrt{g}}}{X_{k}}}\big)_{k \in \N} $ is a Cauchy sequence in $ \func{\chi}{A} $ with respect to $ \Norm{\cdot}_{\st} $ that has a $ \Norm{\cdot}_{\st} $-limit $ X' $, as $ \Pair{\func{\chi}{A}}{\Norm{\cdot}_{\st}} $ is a complete metric space. Hence,
\begin{gather*}
 \lim_{k \to \infty} \big\|{X_{k} - \func{T_{\sqrt{g}}^{- 1}}{X'}}\big\|_{g}
= \lim_{k \to \infty} \big\| T_{\sqrt{g}}\big(X_{k} - \func{T_{\sqrt{g}}^{- 1}}{X'}\big) \big\|_{\st}
= \lim_{k \to \infty} \big\|{\func{T_{\sqrt{g}}}{X_{k}} - X'}\big\|_{\st}
= 0.
\end{gather*}
Therefore, $ \big({\func{\chi}{A}},{\Norm{\cdot}_{g}}\big) $ is a complete metric space, which establishes (1) and (2).

Finally, (3) follows from the definition of $ \Inn_{g} $.
\end{proof}

The following lemma will be needed in Section~\ref{section4}.

\begin{Lem} \label{Norms Associated to Different Matrices Are Equivalent}
Let $ g,h \in \func{M_{n}}{A} $ be positive and invertible. Then
\begin{gather*}
 \frac{1}{\big\|\sqrt{h} \sqrt{g^{- 1}}\big\|_{\func{M_{n}}{A}}} \Norm{\cdot}_{h}
\leq \Norm{\cdot}_{g}
\leq \big\|\sqrt{g} \sqrt{h^{- 1}}\big\|_{\func{M_{n}}{A}} \Norm{\cdot}_{h},
\end{gather*}
so $ \Norm{\cdot}_{g} $ and $ \Norm{\cdot}_{h} $ are equivalent norms on $ \func{\chi}{A} $. In particular,
\begin{gather*}
 \frac{1}{\big\|\sqrt{h}\big\|_{\func{M_{n}}{A}}} \Norm{\cdot}_{h}
\leq \Norm{\cdot}_{\st} \leq \big\|\sqrt{h^{- 1}}\big\|_{\func{M_{n}}{A}} \Norm{\cdot}_{h}.
\end{gather*}
\end{Lem}

\begin{proof}Observe that
\begin{align*}
\forall\, X \in \func{\chi}{A}\colon \quad \Norm{X}_{h}
& = \big\|\func{T_{\sqrt{h}}}{X}\big\|_{\st} \quad \Br{\text{by (2) of Proposition~\ref{Properties of a Riemannian Metric}}} \\
& = \big\|\func{T_{\sqrt{h}} T_{\sqrt{g}}^{- 1} T_{\sqrt{g}}}{X}\big\|_{\st} \\
& = \big\| \func{T_{\sqrt{h} \sqrt{g^{- 1}}} T_{\sqrt{g}}}{X}\big\|_{\st} \\
& \leq \big\| T_{\sqrt{h} \sqrt{g^{- 1}}}\big\|_{\Adj{\func{\chi}{A}}{\Inn_{\st}}} \big\|\func{T_{\sqrt{g}}}{X}\big\|_{\st} \\
& = \big\|T_{\sqrt{h} \sqrt{g^{- 1}}}\big\|_{\Adj{\func{\chi}{A}}{\Inn_{\st}}} \Norm{X}_{g} \quad
 \Br{\text{by (2) of Proposition~\ref{Properties of a Riemannian Metric} again}} \\
& = \big\|\sqrt{h} \sqrt{g^{- 1}}\big\|_{\func{M_{n}}{A}} \Norm{X}_{g}.
\end{align*}
Interchanging $ g $ and $ h $ in the relations above, we get $ \Norm{X}_{g} \leq \big\|\sqrt{h} \sqrt{g^{- 1}}\big\|_{\func{M_{n}}{A}} \Norm{X}_{h} $ for all $ X \in \func{\chi}{A} $, so
\begin{gather*}
\frac{1}{\big\|\sqrt{h} \sqrt{g^{- 1}}\big\|_{\func{M_{n}}{A}}} \Norm{\cdot}_{h}
\leq \Norm{\cdot}_{g} \leq \big\|\sqrt{g} \sqrt{h^{- 1}}\big\|_{\func{M_{n}}{A}} \Norm{\cdot}_{h}
\end{gather*}
as required.

Finally, as $ \Norm{\cdot}_{\st} = \Norm{\cdot}_{\I} $, where $ \I $ is the identity of $ \func{M_{n}}{A} $, the second part comes from letting $ g = \I $.
\end{proof}

It was established in \cite{BEJ} that
\begin{gather} \label{Decomposition of *-Derivations into Their Invariant and Inner Parts}
 \Der{\Theta} = \Set{\sum_{j = 1}^{n} r_{j} \cdot \partial_{j} + \ad{a}}{r_{1},\ldots,r_{n} \in \R ~ \text{and} ~ a \in \ska{\Br{\Smooth{A_{\Theta}}}}},
\end{gather}
where the hypothesis that $ \Theta $ is generically transcendental plays a role. For any $ a,b \in A_{\Theta} $, we have $ \ad{a} = \ad{b} $ if and only if $ a - b \in \C \cdot 1_{A_{\Theta}} $, so if $ \tau $ denotes the faithful tracial state on $ A_{\Theta} $, then~(\ref{Decomposition of *-Derivations into Their Invariant and Inner Parts}) can be rewritten as
\begin{gather*}
 \Der{\Theta} = \Set{\sum_{j = 1}^{n} r_{j} \cdot \partial_{j} + \ad{a - \func{\tau}{a} \cdot 1_{A_{\Theta}}}}
 {r_{1},\ldots,r_{n} \in \R ~ \text{and} ~ a \in \ska{\Br{\Smooth{A_{\Theta}}}}}.
\end{gather*}
Now, for every $ a \in \ska{\Br{A_{\Theta}}} $, the following observations can be made:
\begin{itemize}\itemsep=0pt
\item as $ \func{\tau}{1_{A_{\Theta}}} = 1 $, we have $ \func{\tau}{a - \func{\tau}{a} \cdot 1_{A_{\Theta}}} = 0 $,
\item as $ a $ is skew-adjoint, we have $ a = i b $ for some self-adjoint $ b \in A_{\Theta} $, so
\begin{gather*}
\SqBr{a - \func{\tau}{a} \cdot 1_{A_{\Theta}}}^{\ast} = a^{\ast} - \overline{\func{\tau}{a}} \cdot 1_{A_{\Theta}}.
\end{gather*}
Hence, $ a - \func{\tau}{a} \cdot 1_{A_{\Theta}} $ is skew-adjoint as well.
\end{itemize}
By \cite[Theorem 1.4]{Ro}, the decomposition of every element of $ \Der{\Theta} $ into an $ \R $-linear combination of the $ \partial_{j} $'s and an inner $ \ast $-derivation is unique, so the observations above tell us that $ \Der{\Theta} $ can be linearly parametrized by the normed $ \R $-vector space $ \R^{n} \times \SqBr{\ska{\Br{\Smooth{A_{\Theta}}}} \cap \ker{\tau}} $, whose norm $ \Norm{\cdot} $ we have chosen to be defined by
\begin{gather*}
\forall\, r_{1},\ldots,r_{n} \in \R, ~ \forall\, a \in \ska{\Br{\Smooth{A_{\Theta}}}} \cap \ker{\tau}\colon \quad
\Norm{\Br{r_{1},\ldots,r_{n};a}} \df \func{N}{r_{1},\ldots,r_{n}} + \Norm{a}_{A_{\Theta}}.
\end{gather*}
We may then transfer $ \Norm{\cdot} $ to an $ \R $-vector-space norm $ \Norm{\cdot}_{\Der{\Theta}} $ on $ \Der{\Theta} $ in an obvious way.

For the rest of this paper, we will let $ \chi_{\Theta} \df \func{\chi}{A_{\Theta}} $ and
\begin{gather*}
\Smooth{\chi_{\Theta}} \df \Set{\sum_{j = 1}^{n} a_{j} \bullet \mathbf{e}_{j} \in \chi_{\Theta}}{a_{1},\ldots,a_{n} \in \Smooth{A_{\Theta}}}.
\end{gather*}

\begin{Def}[\cite{Ro}] A \emph{Riemannian metric} is an element of $ \func{M_{n}}{A_{\Theta}} $ that is positive, is invertible, and has all entries in $ \sa{\Br{\Smooth{A_{\Theta}}}} $.
\end{Def}

\begin{Def}[\cite{Ro}] \label{Connection and Levi-Civita Connection} A~\emph{connection} for a Riemannian metric $ g $ is a map $ \nabla\colon \Der{\Theta} \times \Smooth{\chi_{\Theta}} \to \Smooth{\chi_{\Theta}} $ satisfying the following four properties:
\begin{enumerate}\itemsep=0pt
\item[1)] $ \nabla $ is $ \R $-linear in the first variable and $ \C $-linear in the second,
\item[2)] $ \func{\nabla_{\delta}}{a \bullet X} = \func{\delta}{a} \bullet X + a \bullet \func{\nabla_{\delta}}{X} $ for all $ \delta \in \Der{\Theta} $, $ a \in \Smooth{A_{\Theta}} $ and $ X \in \Smooth{\chi_{\Theta}} $ (this property is known as the \emph{Leibniz rule}),
\item[3)] $ \func{\nabla_{\ad{a}}}{X} = a \bullet X $ for all $ a \in \ska{\Br{\Smooth{A_{\Theta}}}} \cap \ker{\tau} $ and $ X \in \Smooth{\chi_{\Theta}} $,
\item[4)] $ \Inner{\func{\nabla_{\partial_{j}}}{\e{k}}}{\e{l}}_{g} $ is self-adjoint for all $ j,k,l \in \SqBr{n} $.
\end{enumerate}
If $ \nabla $ further satisfies the following two properties, then we call it a \emph{Levi-Civita connection} for $ g $:
\begin{enumerate}\itemsep=0pt
\item[5)] $ \func{\nabla_{\partial_{j}}}{\e{k}} = \func{\nabla_{\partial_{k}}}{\e{j}} $ for all $ j,k \in \SqBr{n} $ (this property is known as the \emph{torsion-freeness} of $ \nabla $ (we will not need this property in our proofs),
\item[6)] $\delta \big(\Inner{X}{Y}_{g}\big) = \Inner{\func{\nabla_{\delta}}{X}}{Y}_{g} + \Inner{X}{\func{\nabla_{\delta}}{Y}}_{g} $ for all $ \delta \in \Der{\Theta} $ and $ X,Y \in \Smooth{\chi_{\Theta}} $ (this property is known as the \emph{compatibility} of~$ \nabla $ with~$ g $).
\end{enumerate}
\end{Def}

We now state, for the class of all generically transcendental quantum tori, a noncommutative version of the fundamental theorem of Riemannian geometry.

\begin{Thm}[\cite{Ro}] \label{The Fundamental Theorem of Riemannian Geometry for Generically Transcendental Quantum Tori}
For every Riemannian metric $ g $, there exists a unique Levi-Civita connection $ \nabla^{g} $ for $ g $, which necessarily satisfies the following identity\footnote{This identity is used to prove the classical version of the fundamental theorem of Riemannian geometry.}:
\begin{gather*}
\forall\, j,k,l \in \SqBr{n}\colon \quad
 \big\langle {\func{\nabla^{g}_{\partial_{j}}}{\e{k}}}\,|\,{\e{l}}\big\rangle_{g} = g^{\natural}_{j k l}
\df \frac{1}{2} \cdot \SqBr{\func{\partial_{j}}{g_{k l}} + \func{\partial_{k}}{g_{j l}} - \func{\partial_{l}}{g_{j k}}}.
\end{gather*}
\end{Thm}

\section[Metrized quantum vector bundles over generically transcendental quantum tori]{Metrized quantum vector bundles\\ over generically transcendental quantum tori}\label{section4}

We will first define quantum compact metric spaces and metrized quantum vector bundles. We will then show how to build metrized quantum vector bundles over a generically transcendental quantum torus from Riemannian metrics, using Rosenberg's Levi-Civita connections for these metrics.

\begin{Def}[\cite{L4}]An \emph{admissible function} is a function $ F\colon \CO{0}{\infty}^{4} \to \CO{0}{\infty} $ with the fol\-lowing properties:
\begin{itemize}\itemsep=0pt
\item $ F $ is non-decreasing in each argument,
\item $ w z + x y \leq \func{F}{w,x,y,z} $ for all $ w,x,y,z \in \CO{0}{\infty} $.
\end{itemize}
If $ \func{F}{w,x,y,z} = w z + x y $ for all $ w,x,y,z \in \CO{0}{\infty} $, then we say that it is \emph{Leibniz}.
\end{Def}

\begin{Def}[\cite{L2,L4}] \label{Quasi-Leibniz Quantum Compact Metric Space} Let $ F $ be an admissible function. An \emph{$ F $-quasi-Leibniz quantum compact metric space} is then an ordered pair $ \Pair{A}{\mathsf{L}} $ satisfying the following six properties:
\begin{enumerate}\itemsep=0pt
\item[1)] $ A $ is a unital $ C^{\ast} $-algebra,
\item[2)] $ \L $ is a seminorm defined on a dense Jordan--Lie subalgebra of $ \sa{A} $,
\item[3)] $ \Set{a \in \Dom{\L}}{\func{\L}{a} = 0} = \R \cdot 1_{A} $,
\item[4)] the \emph{Monge--Kantorovich metric} of $ \L $, which is the function $ \mk_{\L}\colon \State{A} \times \State{A} \to \CC{0}{\infty} $ defined by
\begin{gather*}
\forall\, \phi,\psi \in \State{A}\colon \quad
\func{\mk_{\L}}{\phi,\psi} \df \func{\sup}{\Set{\abs{\func{\phi}{a} - \func{\psi}{a}}}{a \in \L^{1}}},
\end{gather*}
metrizes the weak-$ \ast $ topology on $ \State{A} $,
\item[5)] $ \L^{1} $ is $ \Norm{\cdot}_{A} $-closed in $ A $ (equivalently, $ \L^{r} $ is $ \Norm{\cdot}_{A} $-closed in $ A $ for every $ r \in \R_{> 0} $),
\item[6)] for all $ a,b \in \Dom{\L} $, we have
\begin{gather*}
 \func{\max}{\func{\L}{\frac{1}{2} \cdot \Br{a b + b a}},\func{\L}{\frac{1}{2 i} \cdot \Br{a b - b a}}}\leq \func{F}{\Norm{a}_{A},\Norm{b}_{A},\func{\L}{a},\func{\L}{b}}.
\end{gather*}
\end{enumerate}
If $ F $ is implicitly understood, then we simply call $ \Pair{A}{\L} $ a \emph{quasi-Leibniz quantum compact metric space}, and if~$ F $ is Leibniz, then we simply call $ \Pair{A}{\L} $ a \emph{Leibniz quantum compact metric space}. The seminorm $ \L $ is called a \emph{Lipschitz seminorm}, and it is what endows $ A $ with its special metric structure.
\end{Def}

\begin{Eg} The following classes of $ C^{\ast} $-algebras can be endowed with Lipschitz seminorms that turn them into quasi-Leibniz quantum compact metric spaces: AF $ C^{\ast} $-algebras~\cite{AL}, curved noncommutative tori~\cite{L1}, noncommutative solenoids~\cite{LP}, and $ C^{\ast} $-algebras equipped with an ergodic action of a compact group~\cite{Ri1}.
\end{Eg}

\begin{Def}[\cite{L4}]An \emph{admissible triple} is an ordered triple $ \Trip{F}{G}{H} $ with the following properties:
\begin{itemize}\itemsep=0pt
\item $ F $ is an admissible function,
\item $ G\colon \CO{0}{\infty}^{3} \to \CO{0}{\infty} $ and $ H\colon \CO{0}{\infty}^{2} \to \CO{0}{\infty} $,
\item $ G $ and $ H $ are non-decreasing in each of their arguments,
\item $ \Br{x + y} z \leq \func{G}{x,y,z} $ and $ 2 x y \leq \func{H}{x,y} $ for all $ x,y,z \in \CO{0}{\infty} $.
\end{itemize}
\end{Def}

\begin{Def}[\cite{L4}] \label{Metrized Quantum Vector Bundle} Let $ \Trip{F}{G}{H} $ be an admissible triple. An \emph{$ \Trip{F}{G}{H} $-metrized quantum vector bundle} is then an ordered $ 5 $-tuple $ \Br{\X,\Inn,\D,A,\L} $ satisfying the following seven properties:
\begin{enumerate}\itemsep=0pt
\item[1)] $ \Pair{A}{\L} $ is an $ F $-quasi-Leibniz quantum compact metric space,
\item[2)] $ \Pair{\X}{\Inn} $ is a left Hilbert $ A $-module (we denote the associated norm on $ \X $ by $ \Norm{\cdot}_{\Inn} $),
\item[3)] $ \D $ is a norm (not merely a seminorm) defined on a $ \Norm{\cdot}_{\Inn} $-dense $ \C $-linear subspace of $ \X $,
\item[4)] $ \Norm{\zeta}_{\Inn} \leq \func{\D}{\zeta} $ for all $ \zeta \in \Dom{\D} $,
\item[5)] $ \D^{1} $ is $ \Norm{\cdot}_{\Inn} $-compact in $ \X $ (equivalently, $ \D^{r} $ is $ \Norm{\cdot}_{\Inn} $-compact in $ \X $ for every $ r \in \R_{> 0} $),
\item[6)] for all $ a \in \Dom{\L} $ and $ \zeta \in \Dom{\D} $, we have $ a \bullet \zeta \in \Dom{\D} $ and
\begin{gather*}
\func{\D}{a \bullet \zeta} \leq \func{G}{\Norm{a}_{A},\func{\L}{a},\func{\D}{\zeta}},
\end{gather*}
\item[7)] for all $ \zeta,\eta \in \Dom{\D} $, we have $ \func{\Re}{\Inner{\zeta}{\eta}},\func{\Im}{\Inner{\zeta}{\eta}} \in \Dom{\L} $ and
\begin{gather*}
 \func{\max}{\func{\L}{\func{\Re}{\Inner{\zeta}{\eta}}},\func{\L}{\func{\Im}{\Inner{\zeta}{\eta}}}}
\leq \func{H}{\func{\D}{\zeta},\func{\D}{\eta}}.
\end{gather*}
\end{enumerate}
If $ \Trip{F}{G}{H} $ is implicitly understood, then we simply call $ \Br{\X,\Inn,\D,A,\L} $ a \emph{metrized quantum vector bundle}.
\end{Def}

\begin{Eg}In \cite{L4}, Latr\'emoli\`ere constructed metrized quantum vector bundles from actual vector bundles over compact Riemannian manifolds that provide motivating examples for Definition~\ref{Metrized Quantum Vector Bundle}. In the same paper, he also showed that free Hilbert modules over the underlying unital $ C^{\ast} $-algebra of a quasi-Leibniz quantum compact metric space can be turned into metrized quantum vector bundles. The tools developed in~\cite{L4} are then applied in~\cite{L5} to prove the convergence of Heisenberg modules over quantum $ 2 $-tori with respect to the modular Gromov--Hausdorff propinquity.
\end{Eg}

Viewing $ \T^{n} $ as $ \Quot{\R^{n}}{\Z^{n}} $, we may define an $ N $-dependent continuous length function $ \ell $ on $ \T^{n} $ by
\begin{gather*}
\forall\, \bm{s} \in \CO{- \frac{1}{2}}{\frac{1}{2}}^{n}\colon \quad \func{\ell}{\bm{s} + \Z^{n}} \df \func{N}{\bm{s}}.
\end{gather*}
Using $ \ell $, we may then define an $ N $-dependent seminorm $ \L $ on a dense Jordan--Lie subalgebra of~$ \sa{\Br{A_{\Theta}}} $ by
\begin{gather*}
 \Dom{\L} = \left\{{a \in \sa{\Br{A_{\Theta}}}}\,\bigg|\,
 {
 \func{\sup}{\Set{\frac{\Norm{\alph{\bm{t}}{a} - a}_{A_{\Theta}}}{\func{\ell}{\bm{t}}}}{\bm{t} \in \T^{n} \setminus \SSet{\bm{0}}}}
 < \infty }\right\}, \\
\forall\, a \in \Dom{\L}\colon \quad \func{\L}{a} \df \func{\sup}{\Set{\frac{\Norm{\alph{\bm{t}}{a} - a}_{A_{\Theta}}}{\func{\ell}{\bm{t}}}}{\bm{t} \in \T^{n} \setminus \SSet{\bm{0}}}}.
\end{gather*}
The action $ \alpha $ of $ \T^{n} $ on $ A_{\Theta} $ is ergodic, so $ \Pair{A_{\Theta}}{\L} $ is a \emph{Leibniz} quantum compact metric space by \cite[Section~2]{Ri1}. Furthermore, according to \cite[Proposition~8.6]{Ri3}, we have
\begin{gather}
\forall\, a \in \sa{\Br{\Smooth{A_{\Theta}}}}\colon \notag \\
 \func{\L}{a}
= \func{\max}{
 \Set{\Norm{\sum_{j = 1}^{n} r_{j} \cdot \func{\partial_{j}}{a}}_{A_{\Theta}}}
 {r_{1},\ldots,r_{n} \in \R ~ \text{and} ~ \func{N}{r_{1},\ldots,r_{n}} \leq 1}
 }. \label{An Alternative Expression for Rieffel's Lipschitz Seminorm}
\end{gather}
This alternative expression for $ \L $ on $ \sa{\Br{\Smooth{A_{\Theta}}}} $ will play an important role in what is to follow.

\begin{Lem} \label{A Preliminary Inequality} Let $ \delta \in \Der{\Theta} $. Then $ \Norm{\func{\delta}{a}}_{A_{\Theta}} \leq \big[2 \Norm{a}_{A_{\Theta}} + \func{\L}{a}\big] \Norm{\delta}_{\Der{\Theta}} $ for all $ a \in \sa{\Br{\Smooth{A_{\Theta}}}} $.
\end{Lem}

\begin{proof}Write $ \d \delta = \sum_{j = 1}^{n} r_{j} \cdot \partial_{j} + \ad{b} $, where $ r_{1},\ldots,r_{n} \in \R $ and $ b \in \ska{\Br{\Smooth{A_{\Theta}}}} \cap \ker{\tau} $. As
\begin{gather*}
\Norm{\delta}_{\Der{\Theta}} = \func{N}{r_{1},\ldots,r_{n}} + \Norm{b}_{A_{\Theta}},
\end{gather*}
we have $ \func{N}{r_{1},\ldots,r_{n}} \leq \Norm{\delta}_{\Der{\Theta}} $ and $ \Norm{b}_{A_{\Theta}} \leq \Norm{\delta}_{\Der{\Theta}} $. Hence,
\begin{align*}
 \Norm{\func{\delta}{a}}_{A_{\Theta}}
& = \Norm{\sum_{j = 1}^{n} r_{j} \cdot \func{\partial_{j}}{a} + \SqBr{b,a}}_{A_{\Theta}} \\
& \leq \Norm{\sum_{j = 1}^{n} r_{j} \cdot \func{\partial_{j}}{a}}_{A_{\Theta}} + \Norm{\SqBr{b,a}}_{A_{\Theta}} \\
& \leq \Norm{\sum_{j = 1}^{n} r_{j} \cdot \func{\partial_{j}}{a}}_{A_{\Theta}} + 2 \Norm{b}_{A_{\Theta}} \Norm{a}_{A_{\Theta}} \\
& \leq \Norm{\delta}_{\Der{\Theta}} \func{\L}{a} + 2 \Norm{\delta}_{\Der{\Theta}} \Norm{a}_{A_{\Theta}} \quad \Br{\text{by (\ref{An Alternative Expression for Rieffel's Lipschitz Seminorm})}} \\
& = \big[2 \Norm{a}_{A_{\Theta}} + \func{\L}{a}\big] \Norm{\delta}_{\Der{\Theta}}.
\end{align*}
This completes the proof.
\end{proof}

\begin{Def}For every Riemannian metric $ g $, define an $ N $-dependent seminorm $ \OpNorm{\cdot}_{g} $ on $ \Smooth{\chi_{\Theta}} $ by
\begin{gather*}
\forall\, X \in \Smooth{\chi_{\Theta}}\colon \quad \OpNorm{X}_{g}
\df \sup \big(\big\{ {\Norm{\func{\nabla^{g}_{\delta}}{X}}_{g}}\,|\,{\delta \in \Der{\Theta} ~ \text{and} ~ \Norm{\delta}_{\Der{\Theta}} \leq 1}\big\}\big),
\end{gather*}
and an $ N $-dependent norm $ \D_{g} $ on $ \Smooth{\chi_{\Theta}} $ by $ \D_{g} \df \max \big({\Norm{\cdot}_{g},\OpNorm{\cdot}_{g}}\big)$.\footnote{This is the point where Rosenberg's Levi-Civita connections come in.}
\end{Def}

\noindent \textbf{Note:} For the rest of \emph{this section only}, $ g $ denotes a Riemannian metric.

\begin{Prop} \label{The G-Inequality}
Let $ a \in \sa{\Br{\Smooth{A_{\Theta}}}} $ and $ X \in \Smooth{\chi_{\Theta}} $. Then $ a \bullet X \in \Smooth{\chi_{\Theta}} $ and
\begin{gather*}
\DArg{g}{a \bullet X} \leq G \big({\Norm{a}_{A_{\Theta}},\func{\L}{a},\DArg{g}{X}}\big),
\end{gather*}
where $ G\colon \CO{0}{\infty}^{3} \to \CO{0}{\infty} $ is defined by $ \func{G}{x,y,z} \df \Br{3 x + y} z $ for all $ x,y,z \in \CO{0}{\infty} $.
\end{Prop}

\begin{proof}
It is clear from the definition of $ \Smooth{\chi_{\Theta}} $ that $ a \bullet X \in \Smooth{\chi_{\Theta}} $.

Next, by the Leibniz rule,
\begin{gather*}
\forall\, \delta \in \Der{\Theta}\colon \quad
\func{\nabla^{g}_{\delta}}{a \bullet X} = \func{\delta}{a} \bullet X + a \bullet \func{\nabla^{g}_{\delta}}{X}.
\end{gather*}
Hence, we have for all $ \delta \in \Der{\Theta} $ satisfying $ \Norm{\delta}_{\Der{\Theta}} \leq 1 $ that
\begin{align*}
 \Norm{\func{\nabla^{g}_{\delta}}{a \bullet X}}_{g}
& \leq \Norm{\func{\delta}{a} \bullet X}_{g} + \Norm{a \bullet \func{\nabla^{g}_{\delta}}{X}}_{g} \\
& \leq \Norm{\func{\delta}{a}}_{A_{\Theta}} \Norm{X}_{g} + \Norm{a}_{A_{\Theta}} \Norm{\func{\nabla^{g}_{\delta}}{X}}_{g} \\
& \leq \big[2 \Norm{a}_{A_{\Theta}} + \func{\L}{a}\big] \Norm{X}_{g} + \Norm{a}_{A_{\Theta}} \Norm{\func{\nabla^{g}_{\delta}}{X}}_{g} \quad \Br{\text{by Lemma~\ref{A Preliminary Inequality}}} \\
& \leq \big[2 \Norm{a}_{A_{\Theta}} + \func{\L}{a}\big] \Norm{X}_{g} + \Norm{a}_{A_{\Theta}} \OpNorm{X}_{g} \Norm{\delta}_{\Der{\Theta}} \\
& \leq \big[2 \Norm{a}_{A_{\Theta}} + \func{\L}{a}\big] \Norm{X}_{g} + \Norm{a}_{A_{\Theta}} \OpNorm{X}_{g} \\
& \leq \big[2 \Norm{a}_{A_{\Theta}} + \func{\L}{a}\big] \DArg{g}{X} + \Norm{a}_{A_{\Theta}} \DArg{g}{X} \\
& = \big[3 \Norm{a}_{A_{\Theta}} + \func{\L}{a}\big] \DArg{g}{X},
\end{align*}
which immediately yields
\begin{gather*}
\OpNorm{a \bullet X}_{g} \leq \big[3 \Norm{a}_{A_{\Theta}} + \func{\L}{a}\big] \DArg{g}{X}.
\end{gather*}
At the same time, it is straightforward to see that
\begin{gather*}
 \Norm{a \bullet X}_{g} \leq \Norm{a}_{A_{\Theta}} \Norm{X}_{g} \leq \Norm{a}_{A_{\Theta}} \DArg{g}{X} \leq \big[3 \Norm{a}_{A_{\Theta}} + \func{\L}{a}\big] \DArg{g}{X}.
\end{gather*}
Therefore,
\begin{gather*}
 \DArg{g}{a \bullet X} = \max \big({\Norm{a \bullet X}_{g},\OpNorm{a \bullet X}_{g}}\big)
\leq G \big({\Norm{a}_{A_{\Theta}},\func{\L}{a},\DArg{g}{X}}\big)
\end{gather*}
as required.
\end{proof}

\begin{Prop} \label{The H-Inequality}
Let $ X,Y \in \Smooth{\chi_{\Theta}} $. Then $ \Re \big({\Inner{X}{Y}_{g}}\big), \Im \big({\Inner{X}{Y}_{g}}\big) \in \sa{\Br{\Smooth{A_{\Theta}}}} $ and
\begin{gather*}
 \max \big( \L \big( \Re \big(\Inner{X}{Y}_{g}\big)\big), \L \big( \Im \big(\Inner{X}{Y}_{g}\big)\big)\big) \leq H \big(\DArg{g}{X},\DArg{g}{Y}\big),
\end{gather*}
where $ H\colon \CO{0}{\infty}^{2} \to \CO{0}{\infty} $ is defined by $ \func{H}{x,y} \df 2 x y $ for all $ x,y \in \CO{0}{\infty} $.
\end{Prop}

\begin{proof}
It is clear from the definition of $ \Smooth{\chi_{\Theta}} $ that $ \Re\big(\Inner{X}{Y}_{g}\big), \Im\big(\Inner{X}{Y}_{g}\big) \in \sa{\Br{\Smooth{A_{\Theta}}}} $.

Next, choose $ r_{1},\ldots,r_{n} \in \R $ so that $ \func{N}{r_{1},\ldots,r_{n}} \leq 1 $. As $ \nabla^{g} $ is compatible\footnote{Metric compatibility is a requirement in the definition of a Levi-Civita connection.} with $ g $, we have
\begin{align*}
 \Norm{\sum_{j = 1}^{n} r_{j} \cdot \func{\partial_{j}}{\Inner{X}{Y}_{g}}}_{A_{\Theta}}
& = \Norm{
 \Inner{\func{\nabla^{g}_{\sum_{j = 1}^{n} r_{j} \cdot \partial_{j}}}{X}}{Y}_{g} +
 \InnerR{X}{\func{\nabla^{g}_{\sum_{j = 1}^{n} r_{j} \cdot \partial_{j}}}{Y}}_{g}
 }_{A_{\Theta}} \\
& \leq \Norm{\Inner{\func{\nabla^{g}_{\sum_{j = 1}^{n} r_{j} \cdot \partial_{j}}}{X}}{Y}_{g}}_{A_{\Theta}} +
 \Norm{\InnerR{X}{\func{\nabla^{g}_{\sum_{j = 1}^{n} r_{j} \cdot \partial_{j}}}{Y}}_{g}}_{A_{\Theta}} \\
& \leq \Norm{\func{\nabla^{g}_{\sum_{j = 1}^{n} r_{j} \cdot \partial_{j}}}{X}}_{g} \Norm{Y}_{g} +
 \Norm{X}_{g} \Norm{\func{\nabla^{g}_{\sum_{j = 1}^{n} r_{j} \cdot \partial_{j}}}{Y}}_{g} \\
& \leq \OpNorm{X}_{g} \Norm{\sum_{j = 1}^{n} r_{j} \cdot \partial_{j}}_{\Der{\Theta}} \Norm{Y}_{g} +
 \Norm{X}_{g} \OpNorm{Y}_{g} \Norm{\sum_{j = 1}^{n} r_{j} \cdot \partial_{j}}_{\Der{\Theta}} \\
& \leq \OpNorm{X}_{g} \Norm{Y}_{g} + \Norm{X}_{g} \OpNorm{Y}_{g} \quad
 \Br{\text{as $ \func{N}{r_{1},\ldots,r_{n}} \leq 1 $}} \\
& \leq 2 \DArg{g}{X} \DArg{g}{Y}.
\end{align*}
It is an easily-verified fact for any $ C^{\ast} $-algebra $ B $ that $ \Norm{\func{\Re}{b}}_{B},\Norm{\func{\Im}{b}}_{B} \leq \Norm{b}_{B} $ for all $ b \in B $, so
\begin{gather*}
 \Norm{\sum_{j = 1}^{n} r_{j} \cdot \partial_{j} \big( \Re \big(\Inner{X}{Y}_{g}\big)\big)}_{A_{\Theta}}
 = \Norm{\func{\Re}{\sum_{j = 1}^{n} r_{j} \cdot \partial_{j}\big(\Inner{X}{Y}_{g}\big)}}_{A_{\Theta}}
 \leq 2 \DArg{g}{X} \DArg{g}{Y}, \\
 \Norm{\sum_{j = 1}^{n} r_{j} \cdot \partial_{j}\big( \Im \big(\Inner{X}{Y}_{g}\big)\big)}_{A_{\Theta}}
 = \Norm{\func{\Im}{\sum_{j = 1}^{n} r_{j} \cdot \partial_{j}\big(\Inner{X}{Y}_{g}\big)}}_{A_{\Theta}}
 \leq 2 \DArg{g}{X} \DArg{g}{Y}.
\end{gather*}
Therefore, as $ r_{1},\ldots,r_{n} \in \R $ are arbitrary subject to $ \func{N}{r_{1},\ldots,r_{n}} \leq 1 $, an application of (\ref{An Alternative Expression for Rieffel's Lipschitz Seminorm}) yields
\begin{gather*}
 \max\big( \L\big( \Re\big(\Inner{X}{Y}_{g}\big)\big), \L\big( \Im\big(\Inner{X}{Y}_{g}\big)\big)\big) \leq H\big(\DArg{g}{X},\DArg{g}{Y}\big)
\end{gather*}
as required.
\end{proof}

\begin{Prop} \label{A Pre-Compactness Result} $ \D_{g}^{1} $ is $ \Norm{\cdot}_{g} $-pre-compact in $ \chi_{\Theta} $.
\end{Prop}

\begin{proof}
Fix $ X = \d \sum_{j = 1}^{n} a_{j} \bullet \e{j} \in \D_{g}^{1} $, where $ a_{1},\ldots,a_{n} \in \Smooth{A_{\Theta}} $.

Choose $ r_{1},\ldots,r_{n} \in \R $ satisfying $ \func{N}{r_{1},\ldots,r_{n}} \leq 1 $, and let $\delta = \sum\limits_{j = 1}^{n} r_{j} \cdot \partial_{j} $. Then by the Leibniz rule,
\begin{align}
 \func{\nabla^{g}_{\delta}}{X}& = \sum_{j = 1}^{n} \func{\nabla^{g}_{\delta}}{a_{j} \bullet \e{j}}
 = \sum_{j = 1}^{n} \SqBr{\func{\delta}{a_{j}} \bullet \e{j} + a_{j} \bullet \func{\nabla^{g}_{\delta}}{\e{j}}} \nonumber \\
& = \sum_{j = 1}^{n} \func{\delta}{a_{j}} \bullet \e{j} + \sum_{j = 1}^{n} a_{j} \bullet \func{\nabla^{g}_{\delta}}{\e{j}}.\label{Leibniz Rule}
\end{align}
Rearranging the terms from both ends of (\ref{Leibniz Rule}) and then taking the norm of both sides, we obtain
\begin{align*}
 \Norm{\sum_{j = 1}^{n} \func{\delta}{a_{j}} \bullet \e{j}}_{g}
& = \Norm{\func{\nabla^{g}_{\delta}}{X} - \sum_{j = 1}^{n} a_{j} \bullet \func{\nabla^{g}_{\delta}}{\e{j}}}_{g}
 \leq \Norm{\func{\nabla^{g}_{\delta}}{X}}_{g} + \sum_{j = 1}^{n} \Norm{a_{j}}_{A_{\Theta}} \Norm{\func{\nabla^{g}_{\delta}}{\e{j}}}_{g} \\
& \leq \OpNorm{X}_{g} \Norm{\delta}_{\Der{\Theta}} +
 \sum_{j = 1}^{n} \Norm{a_{j}}_{A_{\Theta}} \OpNorm{\e{j}}_{g} \Norm{\delta}_{\Der{\Theta}} \\
& \leq 1 + \sum_{j = 1}^{n} \Norm{a_{j}}_{A_{\Theta}} \OpNorm{\e{j}}_{g} \quad
\big(\text{as $ \OpNorm{X}_{g} \leq \DArg{g}{X} \leq 1 $ and $ \Norm{\delta}_{\Der{\Theta}} \leq 1 $}\big).
\end{align*}
As $ \Norm{X}_{g} \leq \DArg{g}{X} \leq 1 $, an application of Lemma~\ref{Norms Associated to Different Matrices Are Equivalent} yields
\begin{gather*}
 \Norm{\sum_{j = 1}^{n} a_{j} \bullet \e{j}}_{\st}
= \Norm{X}_{\st}
\leq \big\|\sqrt{g^{- 1}}\big\|_{\func{M_{n}}{A_{\Theta}}} \Norm{X}_{g}
\leq \big\|\sqrt{g^{- 1}}\big\|_{\func{M_{n}}{A_{\Theta}}}.
\end{gather*}
Hence, $ \Norm{a_{j}}_{A_{\Theta}} \leq \big\|\sqrt{g^{- 1}}\big\|_{\func{M_{n}}{A_{\Theta}}} $ for all $ j \in \SqBr{n} $, so letting $ M \df \max\big(\big\{{\OpNorm{\e{j}}_{g}}\,|\,{j \in \SqBr{n}}\big\}\big)$ gives us
\begin{gather*}
\Norm{\sum_{j = 1}^{n} \func{\delta}{a_{j}} \bullet \e{j}}_{g} \leq 1 + n M \big\|\sqrt{g^{- 1}}\big\|_{\func{M_{n}}{A_{\Theta}}}.
\end{gather*}
It then follows from another application of Lemma~\ref{Norms Associated to Different Matrices Are Equivalent} that
\begin{align*}
 \Norm{\sum_{j = 1}^{n} \func{\delta}{a_{j}} \bullet \e{j}}_{\st}
& \leq \big\|\sqrt{g^{- 1}}\big\|_{\func{M_{n}}{A_{\Theta}}} \Norm{\sum_{j = 1}^{n} \func{\delta}{a_{j}} \bullet \e{j}}_{g} \\
& \leq \big\|\sqrt{g^{- 1}}\big\|_{\func{M_{n}}{A_{\Theta}}} \Br{1 + n M \big\|\sqrt{g^{- 1}}\big\|_{\func{M_{n}}{A_{\Theta}}}}.
\end{align*}
As $ \func{\Re}{\func{\delta}{a_{j}}} = \func{\delta}{\func{\Re}{a_{j}}} $ and $ \func{\Im}{\func{\delta}{a_{j}}} = \func{\delta}{\func{\Im}{a_{j}}} $ for all $ j \in \SqBr{n} $, we have
\begin{align*}
\forall\, j \in \SqBr{n}\colon \quad
 \Norm{\func{\delta}{\func{\Re}{a_{j}}}}_{A_{\Theta}},\Norm{\func{\delta}{\func{\Im}{a_{j}}}}_{A_{\Theta}}
& \leq \Norm{\func{\delta}{a_{j}}}_{A_{\Theta}} \\
& \leq \big\|\sqrt{g^{- 1}}\big\|_{\func{M_{n}}{A_{\Theta}}} \Br{1 + n M \big\|\sqrt{g^{- 1}}\big\|_{\func{M_{n}}{A_{\Theta}}}}.
\end{align*}
Recalling that $ r_{1},\ldots,r_{n} \in \R $ are arbitrary subject to $ \func{N}{r_{1},\ldots,r_{n}} \leq 1 $, an application of (\ref{An Alternative Expression for Rieffel's Lipschitz Seminorm}) yields
\begin{gather*}
\forall\, j \in \SqBr{n}\colon \quad
 \func{\L}{\func{\Re}{a_{j}}},\func{\L}{\func{\Im}{a_{j}}}
\leq \big\|\sqrt{g^{- 1}}\big\|_{\func{M_{n}}{A_{\Theta}}} \Br{1 + n M \big\|\sqrt{g^{- 1}}\big\|_{\func{M_{n}}{A_{\Theta}}}}.
\end{gather*}
Consequently,
\begin{gather*}
\forall\, j \in \SqBr{n}\colon \\
 \func{\Re}{a_{j}},\func{\Im}{a_{j}}
\in \left\{{a \in \Dom{\L}}\Big|
 { \func{\L}{a},\Norm{a}_{A_{\Theta}}
 \leq \big\|\sqrt{g^{- 1}}\big\|_{\func{M_{n}}{A_{\Theta}}} \Br{1 + n M \big\|\sqrt{g^{- 1}}\big\|_{\func{M_{n}}{A_{\Theta}}}} }\right\},
\end{gather*}
where the object on the right-hand side is a $ \Norm{\cdot}_{A_{\Theta}} $-compact subset of $ A_{\Theta} $ (see \cite[Remark~2.46]{L2}). Hence, as $ X \in \D_{g}^{1} $ is arbitrary, we have shown that there exists a single $ \Norm{\cdot}_{A_{\Theta}} $-compact subset of $ A_{\Theta} $ that contains the $ A_{\Theta} $-coefficients of all elements of $ \D_{g}^{1} $, which implies that $ \D_{g}^{1} $ is $ \Norm{\cdot}_{\st} $-pre-compact in $ \chi_{\Theta} $. However, $ \Norm{\cdot}_{\st} $ and $ \Norm{\cdot}_{g} $ are equivalent by Lemma~\ref{Norms Associated to Different Matrices Are Equivalent}, so $ \D_{g}^{1} $ is $ \Norm{\cdot}_{g} $-pre-compact in $ \chi_{\Theta} $.
\end{proof}

\begin{Def}[\cite{M}] \label{Minkowski Gauge Functional}
Let $ A $ be a normed $ \C $-vector space and $ C $ a circled convex subset of $ A $. The \emph{Minkowski gauge functional} associated to $ C $ is then the function $ p $ that satisfies
\begin{gather*}
\Dom{p} = \Set{a \in A}{\text{there exists an} ~ r \in \R_{> 0} ~ \text{such that} ~ a \in r \cdot C}; \\
\forall\, a \in \Dom{p}\colon \quad \func{p}{a} \df \func{\inf}{\Set{r \in \R_{> 0}}{a \in r \cdot C}}.
\end{gather*}
\end{Def}

Now, let $ \DM{g} $ denote the Minkowski gauge functional associated to $ \Cl{\D_{g}^{1}}{\Norm{\cdot}_{g}} $. Our objective is to show that $ \big(\chi_{\Theta},\Inn_{g},\DM{g},A_{\Theta},\L\big) $ is an $ \Trip{F}{G}{H} $-metrized quantum vector bundle, where~$ F $ is Leibniz, and $ G $ and $ H $ are defined as in Propositions~\ref{The G-Inequality} and~\ref{The H-Inequality} respectively. Before proceeding further, let us first collect some facts about Minkowski gauge functionals.

\begin{Lem} \label{Some Facts About Minkowski Gauge Functionals}
Let $ A $ be a normed $ \C $-vector space, $ B $ a $ \C $-linear subspace of $ A $, and $ L $ a seminorm on $ B $. Denote by $ L_{\M} $ the Minkowski gauge functional associated to $ \Cl{L^{1}}{\Norm{\cdot}_{A}} $. Then the following statements hold:
\begin{enumerate}\itemsep=0pt
\item[{\rm 1)}] $ L_{\M}^{r} = \Cl{L^{r}}{\Norm{\cdot}_{A}} $ for all $ r \in \R_{> 0} $,
\item[{\rm 2)}] $ B \subseteq \Dom{L_{\M}} $ and $ \func{L_{\M}}{b} \leq \func{L}{b} $ for all $ b \in B $,
\item[{\rm 3)}] if $ a \in \Dom{L_{\M}} $, then there is a sequence $ \Seq{b_{k}}{k \in \N} $ in $ B $ such that
\begin{itemize}\itemsep=0pt
\item $ \func{L}{b_{k}} \to \func{L_{\M}}{a} $ and
\item $ b_{k} \to a $ with respect to $ \Norm{\cdot}_{A} $,
\end{itemize}
\item[{\rm 4)}] $ L \subseteq L_{\M} $ if $ L $ is lower-semicontinuous on $ \Pair{B}{\Norm{\cdot}_{B}} $, where $ \Norm{\cdot}_{B} $ denotes the restriction of~$ \Norm{\cdot}_{A} $ to~$ B $.
\end{enumerate}
\end{Lem}

\begin{proof} See \cite[Proposition 4.4]{Ri2} and \cite[Section~3]{Ri4}.
\end{proof}

With Lemma~\ref{Some Facts About Minkowski Gauge Functionals}, we can make the following observations:
\begin{itemize}\itemsep=0pt
\item
$ \DM{g} $ satisfies property (3) of Definition~\ref{Metrized Quantum Vector Bundle}.

We must first establish that $ \DM{g} $ is a norm on its domain. Let $ X \in \Dom{\DM{g}} $ satisfy $ \DMArg{g}{X} = 0 $. Then (3) of Lemma~\ref{Some Facts About Minkowski Gauge Functionals} tells us that there is a sequence $ \Seq{X_{k}}{k \in \N} $ in $ \Smooth{\chi_{\Theta}} $ such that $ \DArg{g}{X_{k}} \to \DMArg{g}{X} = 0 $ and $ X_{k} \to X $ with respect to $ \Norm{\cdot}_{g} $, but $ \Norm{X_{k}}_{g} \leq \DArg{g}{X_{k}} $ by construction for every $ k \in \N $, so $ \Norm{X_{k}}_{g} \to 0 $, or equivalently, $ X_{k} \to 0_{\chi_{\Theta}} $ with respect to $ \Norm{\cdot}_{g} $. Therefore, $ X = 0_{\chi_{\Theta}} $, but as Minkowski gauge functionals are already absolutely homogeneous, subadditive and non-negative, $ \DM{g} $ is indeed a norm on its domain.

Next, we have by (2) of Lemma~\ref{Some Facts About Minkowski Gauge Functionals} that $ \Smooth{\chi_{\Theta}} \subseteq \Dom{\DM{g}} $. As $ \Smooth{\chi_{\Theta}} $ is a $ \Norm{\cdot}_{g} $-dense $ \C $-linear subspace of $ \chi_{\Theta} $, so is $ \Dom{\DM{g}} $.

\item $ \DM{g} $ satisfies property (4) of Definition~\ref{Metrized Quantum Vector Bundle}.

Let $ X \in \Dom{\DM{g}} $. Then (3) of Lemma~\ref{Some Facts About Minkowski Gauge Functionals} tells us that there is a sequence $ \Seq{X_{k}}{k \in \N} $ in $ \Smooth{\chi_{\Theta}} $ such that $ \DArg{g}{X_{k}} \to \DMArg{g}{X} $ and $ X_{k} \to X $ with respect to $ \Norm{\cdot}_{g} $, but, as before, $ \Norm{X_{k}}_{g} \leq \DArg{g}{X_{k}} $ by construction for every $ k \in \N $, so $ \Norm{X}_{g} \leq \DMArg{g}{X} $.

\item $ \DM{g} $ satisfies property (5) of Definition~\ref{Metrized Quantum Vector Bundle}.

By (1) of Lemma~\ref{Some Facts About Minkowski Gauge Functionals}, we have $ \DM{g}^{1} = \Cl{\D_{g}^{1}}{\Norm{\cdot}_{g}} $, so $ \DM{g}^{1} $ is $ \Norm{\cdot}_{g} $-compact in $ \chi_{\Theta} $ by Proposition~\ref{A Pre-Compactness Result}.
\end{itemize}

It remains to verify properties (6) and (7) of Definition~\ref{Metrized Quantum Vector Bundle}, which we now turn our attention to.

\begin{Prop} Let $ a \in \Dom{\L} $ and $ X \in \Dom{\DM{g}} $. Then $ a \bullet X \in \Dom{\DM{g}} $ and
\begin{gather*}
\DMArg{g}{a \bullet X} \leq G\big(\Norm{a}_{A_{\Theta}},\func{\L}{a},\DMArg{g}{X}\big),
\end{gather*}
where $ G\colon \CO{0}{\infty}^{3} \to \CO{0}{\infty} $ is defined as in Proposition~{\rm \ref{The G-Inequality}}.
\end{Prop}

\begin{proof}By (3) of Lemma~\ref{Some Facts About Minkowski Gauge Functionals}, there is a sequence $ \Seq{X_{k}}{k \in \N} $ in $ \Smooth{\chi_{\Theta}} $ such that
\begin{itemize}\itemsep=0pt
\item $ \DArg{g}{X_{k}} \to \DMArg{g}{X} $ and
\item $ X_{k} \to X $ with respect to $ \Norm{\cdot}_{g} $.
\end{itemize}
Let $ \mu $ denote the normalized Haar measure on $ \T^{n} $. It is a well-known result of the theory of smooth vectors for a strongly continuous Lie-group action on a $ C^{\ast} $-algebra that a net $ \Seq{f_{\nu}}{\nu \in \mathcal{N}} $ in $ \CcS{\T^{n}} $ exists such that
\begin{enumerate}\itemsep=0pt
\item[(a)] $ f_{\nu} $ is non-negative and $\d \Int{\T^{n}}{\func{f_{\nu}}{\bm{s}}}{\mu}{\bm{s}} = 1 $ for all $ \nu \in \mathcal{N} $,
\item[(b)] $\d a_{\nu} \df \Int{\T^{n}}{\func{f_{\nu}}{\bm{s}} \cdot \alph{\bm{s}}{a}}{\mu}{\bm{s}} \in \Smooth{A_{\Theta}} $ for all $ \nu \in \mathcal{N} $, and
\item[(c)] $ a_{\nu} \to a $ with respect to $ \Norm{\cdot}_{A_{\Theta}} $.
\end{enumerate}
In fact, as $ f_{\nu} $ is non-negative for all $ \nu \in \mathcal{N} $ and as self-adjoint elements of a $ C^{\ast} $-algebra are preserved under $ \ast $-homomorphisms, (b) can be strengthened to say that $ a_{\nu} \in \sa{\Br{\Smooth{A_{\Theta}}}} $ for all $ \nu \in \mathcal{N} $. Proceeding, we have for all $ \nu \in \mathcal{N} $ and $ \bm{t} \in \T^{n} \setminus \SSet{\bm{0}} $ that
\begin{align*}
 \frac{\Norm{\alph{\bm{t}}{a_{\nu}} - a_{\nu}}_{A_{\Theta}}}{\func{\ell}{\bm{t}}}
& = \frac{\d
 \Norm{
 \alph{\bm{t}}{\Int{\T^{n}}{\func{f_{\nu}}{\bm{s}} \cdot \alph{\bm{s}}{a}}{\mu}{\bm{s}}} -
 \Int{\T^{n}}{\func{f_{\nu}}{\bm{s}} \cdot \alph{\bm{s}}{a}}{\mu}{\bm{s}}
 }_{A_{\Theta}}
 }
 {\func{\ell}{\bm{t}}} \\
& = \frac{\d
 \Norm{
 \Int{\T^{n}}{\func{f_{\nu}}{\bm{s}} \cdot \alph{\bm{t} + \bm{s}}{a}}{\mu}{\bm{s}} -
 \Int{\T^{n}}{\func{f_{\nu}}{\bm{s}} \cdot \alph{\bm{s}}{a}}{\mu}{\bm{s}}
 }_{A_{\Theta}}
 }
 {\func{\ell}{\bm{t}}} \\
& = \frac{\d
 \Norm{
 \Int{\T^{n}}{\func{f_{\nu}}{\bm{s}} \cdot \alph{\bm{s} + \bm{t}}{a}}{\mu}{\bm{s}} -
 \Int{\T^{n}}{\func{f_{\nu}}{\bm{s}} \cdot \alph{\bm{s}}{a}}{\mu}{\bm{s}}
 }_{A_{\Theta}}
 }
 {\func{\ell}{\bm{t}}} \\
& = \frac{\Norm{\d \Int{\T^{n}}{\func{f_{\nu}}{\bm{s}} \cdot \alph{\bm{s}}{\alph{\bm{t}}{a} - a}}{\mu}{\bm{s}}}_{A_{\Theta}}}
 {\func{\ell}{\bm{t}}} \\
& \leq \frac{\d \Int{\T^{n}}{\Norm{\func{f_{\nu}}{\bm{s}} \cdot \alph{\bm{s}}{\alph{\bm{t}}{a} - a}}_{A_{\Theta}}}{\mu}{\bm{s}}}
 {\func{\ell}{\bm{t}}} \\
& = \frac{\d \Int{\T^{n}}{\func{f_{\nu}}{\bm{s}} \Norm{\alph{\bm{s}}{\alph{\bm{t}}{a} - a}}_{A_{\Theta}}}{\mu}{\bm{s}}}{\func{\ell}{\bm{t}}}
 \quad
 \Br{\text{as $ f_{\nu} $ is non-negative}} \\
& = \frac{\d \Int{\T^{n}}{\func{f_{\nu}}{\bm{s}} \Norm{\alph{\bm{t}}{a} - a}_{A_{\Theta}}}{\mu}{\bm{s}}}{\func{\ell}{\bm{t}}} \quad
 \Br{\text{as $ \alpha_{\bm{s}} $ is isometric}} \\
& = \frac{\Norm{\alph{\bm{t}}{a} - a}_{A_{\Theta}}}{\func{\ell}{\bm{t}}} \quad
 \Br{\text{as $ \Int{\T^{n}}{\func{f_{\nu}}{\bm{s}}}{\mu}{\bm{s}} = 1 $}}.
\end{align*}
Consequently, $ a_{\nu} \in \Dom{\L} $ and $ \func{\L}{a_{\nu}} \leq \func{\L}{a} $ for all $ \nu \in \mathcal{N} $. Invoking Proposition~\ref{The G-Inequality}, we get
\begin{align*}
\forall\, \nu \in \mathcal{N}, ~ \forall\, k \in \N\colon \quad
 \DArg{g}{a_{\nu} \bullet X_{k}}& \leq G\big(\Norm{a_{\nu}}_{A_{\Theta}},\func{\L}{a_{\nu}},\DArg{g}{X_{k}}\big)\\
&\leq G\big(\Norm{a_{\nu}}_{A_{\Theta}},\func{\L}{a},\DArg{g}{X_{k}}\big),
\end{align*}
where the last inequality is due to the fact that $ G $ is non-decreasing in each argument. By (2) of Lemma~\ref{Some Facts About Minkowski Gauge Functionals}, $ Y \in \Dom{\DM{g}} $ and $ \DMArg{g}{Y} \leq \DArg{g}{Y} $ for all $ Y \in \Smooth{\chi_{\Theta}} $, so
\begin{gather*}
\forall\, \nu \in \mathcal{N}, ~ \forall\, k \in \N\colon \quad
\DMArg{g}{a_{\nu} \bullet X_{k}} \leq G\big(\Norm{a_{\nu}}_{A_{\Theta}},\func{\L}{a},\DArg{g}{X_{k}}\big).
\end{gather*}
Let $ \epsilon > 0 $. As the right-hand side of this inequality converges to $ G\big(\Norm{a}_{A_{\Theta}},\func{\L}{a},\DMArg{g}{X}\big) $ by the continuity of $ G $, we find for all $ \nu \in \mathcal{N} $ and $ k \in \N $ sufficiently large that
\begin{gather*}
\DMArg{g}{a_{\nu} \bullet X_{k}} \leq \func{G}{\Norm{a}_{A_{\Theta}},\func{\L}{a},\DMArg{g}{X}} + \epsilon.
\end{gather*}
As $ a_{\nu} \bullet X_{k} \to a \bullet X $ with respect to $ \Norm{\cdot}_{g} $, it follows from the closure clause implicit in (1) of Lemma~\ref{Some Facts About Minkowski Gauge Functionals} that $ a \bullet X \in \Dom{\DM{g}} $ and
\begin{gather*}
\DMArg{g}{a \bullet X} \leq G\big(\Norm{a}_{A_{\Theta}},\func{\L}{a},\DMArg{g}{X}\big) + \epsilon.
\end{gather*}
However, $ \epsilon > 0 $ is arbitrary, so
\begin{gather*}
\DMArg{g}{a \bullet X} \leq G \big(\Norm{a}_{A_{\Theta}},\func{\L}{a},\DMArg{g}{X}\big)
\end{gather*}
as required.
\end{proof}

\begin{Prop}Let $ X,Y \in \Dom{\DM{g}} $. Then $ \Re\big(\Inner{X}{Y}_{g}\big), \Im\big(\Inner{X}{Y}_{g}\big) \in \Dom{\L} $ and
\begin{gather*}
 \max\big( \L\big( \Re\big(\Inner{X}{Y}_{g}\big)\big), \L\big( \Im\big(\Inner{X}{Y}_{g}\big)\big)\big) \leq H\big(\DMArg{g}{X},\DMArg{g}{Y}\big),
\end{gather*}
where $ H\colon \CO{0}{\infty}^{2} \to \CO{0}{\infty} $ is defined as in Proposition~{\rm \ref{The H-Inequality}}.
\end{Prop}

\begin{proof}By (3) of Lemma~\ref{Some Facts About Minkowski Gauge Functionals}, there are sequences $ \Seq{X_{k}}{k \in \N} $ and $ \Seq{Y_{k}}{k \in \N} $ in $ \Smooth{\chi_{\Theta}} $ such that
\begin{itemize}\itemsep=0pt
\item
$ \DArg{g}{X_{k}} \to \DMArg{g}{X} $ and $ \DArg{g}{Y_{k}} \to \DMArg{g}{Y} $, and

\item
$ X_{k} \to X $ and $ Y_{k} \to Y $ with respect to $ \Norm{\cdot}_{g} $.
\end{itemize}
We already have from Proposition~\ref{The H-Inequality} that
\begin{gather*}
\forall\, k \in \N\colon \quad
 \max\big( \L\big( \Re\big(\Inner{X_{k}}{Y_{k}}_{g}\big)\big), \L\big( \Im\big(\Inner{X_{k}}{Y_{k}}_{g}\big)\big)\big)
\leq H\big(\DArg{g}{X_{k}},\DArg{g}{Y_{k}}\big).
\end{gather*}
Let $ \epsilon > 0 $. As the right-hand side of this inequality converges to $ \func{H}{\DMArg{g}{X},\DMArg{g}{Y}} $ by the continuity of $ H $, we find for all $ k \in \N $ sufficiently large that
\begin{gather*}
 \max\big( \L\big( \Re\big(\Inner{X_{k}}{Y_{k}}_{g}\big)\big), \L\big( \Im\big(\Inner{X_{k}}{Y_{k}}_{g}\big)\big)\big)
\leq H\big(\DMArg{g}{X},\DMArg{g}{Y}\big) + \epsilon.
\end{gather*}
As $ \Re\big(\Inner{X_{k}}{Y_{k}}_{g}\big) \to \Re \big(\Inner{X}{Y}_{g}\big) $ and $ \Im\big(\Inner{X_{k}}{Y_{k}}_{g}\big) \to \Im \big(\Inner{X}{Y}_{g}\big) $ with respect to $ \Norm{\cdot}_{A_{\Theta}} $, the closure clause implicit in property~(5) of Definition~\ref{Quasi-Leibniz Quantum Compact Metric Space} says that $ \Re\big(\Inner{X}{Y}_{g}\big), \Im\big(\Inner{X}{Y}_{g}\big) \in \Dom{\L} $ and
\begin{gather*}
 \max\big( \L\big( \Re\big(\Inner{X}{Y}_{g}\big)\big), \L\big( \Im\big(\Inner{X}{Y}_{g}\big)\big)\big) \leq H\big(\DMArg{g}{X},\DMArg{g}{Y}\big) + \epsilon.
\end{gather*}
However, $ \epsilon > 0 $ is arbitrary, so
\begin{gather*}
 \max\big( \L\big( \Re\big(\Inner{X}{Y}_{g}\big)\big), \L\big( \Im\big(\Inner{X}{Y}_{g}\big)\big)\big) \leq H\big(\DMArg{g}{X},\DMArg{g}{Y} \big)
\end{gather*}
as required.
\end{proof}

\begin{Thm}
$ \big(\chi_{\Theta},\Inn_{g},\DM{g},A_{\Theta},\L\big)$ is an $ \Trip{F}{G}{H} $-metrized quantum vector bundle, where $ F $ is Leibniz, and $ G $ and $ H $ are defined as in Propositions~{\rm \ref{The G-Inequality}} and~{\rm \ref{The H-Inequality}} respectively.
\end{Thm}

\begin{Rmk}We could have worked with more general $ C^{\ast} $-algebras in this paper, but here are reasons why we focus only on generically transcendental quantum tori:
\begin{enumerate}\itemsep=0pt
\item[1)] quantum tori are a convenient source of compact quantum metric spaces,
\item[2)] the norm $ \Norm{\cdot}_{\Der{\Theta}} $ on the space of $ \ast $-derivations on $ \Smooth{A_{\Theta}} $ enables us to conveniently prove the inequalities necessary to achieve the objective of this paper,
\item[3)] by Theorem~\ref{The Fundamental Theorem of Riemannian Geometry for Generically Transcendental Quantum Tori}, we have a \emph{single} parameter for the $ \D $-norms on our metrized quantum vector bundles.
\end{enumerate}
\end{Rmk}

\section[A distance-zero result for the modular Gromov--Hausdorff propinquity]{A distance-zero result\\ for the modular Gromov--Hausdorff propinquity}\label{section5}

In \cite{L4}, Latr\'emoli\`ere introduced the modular Gromov--Hausdorff propinquity as a means of measuring, for an admissible triple $ \Trip{F}{G}{H} $, how close two $ \Trip{F}{G}{H} $-metrized quantum vector bundles are to each other, in the sense of how close they are to being isomorphic, in terms of their Hilbert-$ C^{\ast} $-module structures and their metric structures.

Our aim in this section is to prove the following zero-distance result.

\begin{Thm} \label{Second Main Theorem}
Let $ g $ be a Riemannian metric, and let $ r,s \in \R_{> 0} $. Then $ \ModProp{F}{G}{H}{\Omega_{r}}{\Omega_{s}} = 0 $, where
\begin{gather*}
\Omega_{r} \df \big(\chi_{\Theta},\Inn_{r \cdot g},\DM{r \cdot g},A_{\Theta},\L\big)
\qquad \text{and} \qquad
\Omega_{s} \df \big(\chi_{\Theta},\Inn_{s \cdot g},\DM{s \cdot g},A_{\Theta},\L\big),
\end{gather*}
$ F $ is Leibniz, and $ G $ and $ H $ are defined as in Propositions~{\rm \ref{The G-Inequality}} and {\rm \ref{The H-Inequality}} respectively.
\end{Thm}

\begin{Def}[\cite{L4}] Let $ B $ be a unital $ C^{\ast} $-algebra, and let $ \mathfrak{b} \in \sa{B} $. The \emph{$ 1 $-level set} of $ \mathfrak{b} $ is then the subset $ \State{B|\mathfrak{b}} $ of $ \State{B} $ defined by
\begin{gather*}
 \State{B|\mathfrak{b}}
\df \Set{\phi \in \State{B}}
 {
 \func{\phi}{\Br{1_{B} - \mathfrak{b}}^{\ast} \Br{1_{B} - \mathfrak{b}}}
 = 0
 = \func{\phi}{\Br{1_{B} - \mathfrak{b}} \Br{1_{B} - \mathfrak{b}}^{\ast}}
 }.
\end{gather*}
\end{Def}

\begin{Def}[\cite{L4}]Let $ \Trip{F}{G}{H} $ be an admissible triple. An $ \Trip{F}{G}{H} $-\emph{modular bridge} is then an ordered $ 9 $-tuple $ \gamma = \Br{\Omega_{1},\Omega_{2},B,\mathfrak{b},\pi_{1},\pi_{2},\mathcal{N},\alpha,\beta} $ satisfying the following six properties:
\begin{enumerate}\itemsep=0pt
\item[1)] $ \Omega_{1} $ and $ \Omega_{2} $ are $ \Trip{F}{G}{H} $-metrized quantum vector bundles,\footnote{\textbf{Note.} From now on, we will write $ \Omega_{1} = \Br{\X_{1},\Inn_{1},\D_{1},A_{1},\L_{1}} $ and $ \Omega_{2} = \Br{\mathsf{X}_{2},\Inn_{2},\D_{2},A_{2},\L_{2}} $.}
\item[2)] $ B $ is a unital $ C^{\ast} $-algebra,
\item[3)] $ \mathfrak{b} $ is an element of $ \sa{B} $, called the \emph{pivot}, such that $ \State{B|\mathfrak{b}} \neq \varnothing $ and $ \Norm{\mathfrak{b}}_{B} = 1 $,
\item[4)] $ \pi_{1} $ and $ \pi_{2} $ are, respectively, unital $ \ast $-monomorphisms from $ A_{1} $ and $ A_{2} $ to $ B $,
\item[5)] $ \mathcal{N} $ is a non-empty set,
\item[6)] $ \alpha $ and $ \beta $ are, respectively, functions from $ \mathcal{N} $ to $ \D_{1}^{1} $ and $ \D_{2}^{1} $.
\end{enumerate}
The \emph{domain} of $ \gamma $, denoted by $ \Dom{\gamma} $, is defined to be $ \Omega_{1} $, and the \emph{co-domain} of $ \gamma $, denoted by $ \CoDom{\gamma} $, is defined to be $ \Omega_{2} $. We say that $\Br{\Omega_{1},\Omega_{2},B,\mathfrak{b}, \pi_{1},\pi_{2},\mathcal{N},\alpha,\beta} $ is a \emph{modular bridge} from $ \Omega_{1} $ to $ \Omega_{2} $.
\end{Def}

A modular bridge between two metrized quantum vector bundles yields useful numerical quantities that indicate how close the metrized quantum vector bundles are to each other. Before defining these quantities, we require a preliminary definition.

\begin{Def}[\cite{L4}]Let $ \Omega = \Br{\X,\Inn,\D,A,\L} $ be a metrized quantum vector bundle. The \emph{modular Monge--Kantorovich metric} of $ \Omega $ is then the metric\footnote{This is a legitimate metric, thanks to property~(5) of Definition~\ref{Metrized Quantum Vector Bundle}.} $ \k_{\Omega} $ on $ \X $ defined by
\begin{gather*}
\forall\, \zeta,\eta \in \X\colon \quad
\func{\k_{\Omega}}{\zeta,\eta} \df \sup \big(\big\{{\Norm{\Inner{\zeta - \eta}{\theta}}_{\Inn}}\,|\,{\theta \in \D^{1}}\big\}\big).
\end{gather*}
\end{Def}

\begin{Def}[\cite{L4}] \label{Numerical Quantities Associated to a Modular Bridge} Let $ \Trip{F}{G}{H} $ be an admissible triple and
\begin{gather*}
\gamma = \Br{\Omega_{1},\Omega_{2},B,\mathfrak{b},\pi_{1},\pi_{2},\mathcal{N},\alpha,\beta}
\end{gather*}
an $ \Trip{F}{G}{H} $-modular bridge.
\begin{enumerate}\itemsep=0pt
\item The \emph{bridge seminorm} of $ \gamma $ is the seminorm $ \bn_{\gamma} $ on $ A_{1} \oplus A_{2} $ defined by
\begin{gather*}
\forall\, a_{1} \in A_{1}, ~ \forall\, a_{2} \in A_{2}\colon \quad
\bnArg{\gamma}{a_{1}}{a_{2}} \df \Norm{\func{\pi_{1}}{a_{1}} \mathfrak{b} - \mathfrak{b} \func{\pi_{2}}{a_{2}}}_{B}.
\end{gather*}

\item The \emph{basic reach} $ \func{\rho_{\flat}}{\gamma} $ of $ \gamma $ is the Hausdorff distance, in the seminormed space $ \Pair{A_{1} {\oplus} A_{2}}{\bn_{\gamma}} $, between the embedded images of $ \Dom{\L_{1}} $ and $ \Dom{\L_{2}} $, i.e.,
\begin{gather*}
 \func{\rho_{\flat}}{\gamma}
\df \func{\max}{
 \adjustlimits \sup_{\substack{a_{1} \in \Dom{\L_{1}} \\ \func{\L_{1}}{a_{1}} \leq 1}}
 \inf_{\substack{a_{2} \in \Dom{\L_{2}} \\ \func{\L_{2}}{a_{2}} \leq 1}}
 \bnArg{\gamma}{a_{1}}{a_{2}},
 \adjustlimits \sup_{\substack{a_{2} \in \Dom{\L_{2}} \\ \func{\L_{2}}{a_{2}} \leq 1}}
 \inf_{\substack{a_{1} \in \Dom{\L_{1}} \\ \func{\L_{1}}{a_{1}} \leq 1}}
 \bnArg{\gamma}{a_{1}}{a_{2}}
 }.
\end{gather*}

\item
The \emph{height} of $ \gamma $ is the non-negative quantity $ \func{\varsigma}{\gamma} $ defined by
\begin{gather*}
 \func{\varsigma}{\gamma}
\df \func{\max}{
 \begin{matrix}
 \Haus{\mk_{\L_{1}}}{\State{A_{1}}}{\Set{\phi \circ \pi_{1}}{\phi \in \State{B|\mathfrak{b}}}} \\
 \Haus{\mk_{\L_{2}}}{\State{A_{2}}}{\Set{\psi \circ \pi_{2}}{\psi \in \State{B|\mathfrak{b}}}}
 \end{matrix}
 },
\end{gather*}
where
\begin{itemize}\itemsep=0pt
\item $ \mathsf{Haus}_{\mk_{\L_{1}}} $ denotes the $ \mk_{\L_{1}} $-induced Hausdorff distance between subsets of $ \State{A_{1}} $, and

\item $ \mathsf{Haus}_{\mk_{\L_{2}}} $ denotes the $ \mk_{\L_{2}} $-induced Hausdorff distance between subsets of $ \State{A_{2}} $.
\end{itemize}

\item The \emph{deck seminorm} of $ \gamma $ is the seminorm $ \dn_{\gamma} $ on $ \X_{1} \oplus \X_{2} $ defined by
\begin{gather*}
\forall\, \zeta \in \X_{1}, ~ \forall\, \eta \in \X_{2}\colon \\
 \dnArg{\gamma}{\zeta}{\eta}
\df \func{\sup}{
 \Set{
 \bnArg{\gamma}{\Inner{\zeta}{\func{\alpha}{\nu}}_{1}}{\Inner{\eta}{\func{\beta}{\nu}}_{2}},
 \bnArg{\gamma}{\Inner{\func{\alpha}{\nu}}{\zeta}_{1}}{\Inner{\func{\beta}{\nu}}{\eta}_{2}}
 }
 {\nu \in \mathcal{N}}
 }.
\end{gather*}

\item The \emph{modular reach} of $ \gamma $ is the non-negative quantity $ \func{\rho^{\sharp}}{\gamma} $ defined by
\begin{gather*}
\func{\rho^{\sharp}}{\gamma} \df \func{\sup}{\Set{\dnArg{\gamma}{\func{\alpha}{\nu}}{\func{\beta}{\nu}}}{\nu \in \mathcal{N}}}.
\end{gather*}

\item
The \emph{imprint} of $ \gamma $ is the non-negative quantity $ \func{\varpi}{\gamma} $ defined by
\begin{gather*}
\func{\varpi}{\gamma} \df \max \big(\Haus{\k_{\Omega_{1}}}{\Range{\alpha}}{\D_{1}^{1}},\Haus{\k_{\Omega_{2}}}{\Range{\beta}}{\D_{2}^{1}}\big),
\end{gather*}
where
\begin{itemize}\itemsep=0pt
\item $ \mathsf{Haus}_{\k_{\Omega_{1}}} $ denotes the $ \k_{\Omega_{1}} $-induced Hausdorff distance between subsets of $ \X_{1} $, and
\item $ \mathsf{Haus}_{\k_{\Omega_{2}}} $ denotes the $ \k_{\Omega_{2}} $-induced Hausdorff distance between subsets of $ \X_{2} $.
\end{itemize}

\item The \emph{reach} of $ \gamma $ is the non-negative quantity $ \func{\rho}{\gamma} $ defined by
\begin{gather*}
\func{\rho}{\gamma} \df \max \big(\func{\rho_{\flat}}{\gamma},\func{\rho^{\sharp}}{\gamma} + \func{\varpi}{\gamma}\big).
\end{gather*}

\item The \emph{length} of $ \gamma $ is the non-negative quantity $ \func{\lambda}{\gamma} $ defined by
\begin{gather*}
\func{\lambda}{\gamma} \df \func{\max}{\func{\varsigma}{\gamma},\func{\rho}{\gamma}}.
\end{gather*}
\end{enumerate}
\end{Def}

\begin{Rmk}Lemma 4.19 of \cite{L4} establishes that the numerical quantities defined in Definition~\ref{Numerical Quantities Associated to a Modular Bridge} are finite.
\end{Rmk}

\begin{Def}[\cite{L4}] The \emph{modular Gromov--Hausdorff propinquity}, for an admissible triple $ \Trip{F}{G}{H} $, is the (class) function $ \Lambda^{\mod}_{F,G,H} $ from the class of all ordered pairs of $ \Trip{F}{G}{H} $-metrized quantum vector bundles to $ \CO{0}{\infty} $ defined by
\begin{gather*}
 \ModProp{F}{G}{H}{\Omega_{1}}{\Omega_{2}}
\df \func{\inf}{
 \left\{{\sum_{j = 1}^{m} \func{\lambda}{\gamma_{j}}}\left|\,
 {
 \begin{matrix}
 \Seq{\gamma_{j}}{j \in \SqBr{m}} ~ \text{is a finite sequence} \\\ \text{of} ~ \Trip{F}{G}{H}\text{-modular bridges}, \\
 \Dom{\gamma_{1}} = \Omega_{1} ~ \text{and} ~ \CoDom{\gamma_{m}} = \Omega_{2}, \\
 \Dom{\gamma_{j + 1}} = \CoDom{\gamma_{j}} ~ \text{for all} ~ j \in \SqBr{m - 1}
 \end{matrix} } \right.\right\}}
\end{gather*}
for all $ \Trip{F}{G}{H} $-metrized quantum vector bundles $ \Omega_{1} $ and $ \Omega_{2} $.
\end{Def}

We now define a full quantum isometry between metrized quantum vector bundles.

\begin{Def}[\cite{L4}]
Let $ \Trip{F}{G}{H} $ be an admissible triple. Let $ \Omega_{1} $ and $ \Omega_{2} $ be $ \Trip{F}{G}{H} $-metrized quantum vector bundles. A \emph{full quantum isometry} from $ \Omega_{1} $ to $ \Omega_{2} $ is then an ordered pair $ \Pair{\mathscr{L}}{\mathscr{D}} $ with the following properties:
\begin{itemize}\itemsep=0pt
\item $ \mathscr{L} $ is a unital $ \ast $-isomorphism from $ A_{1} $ to $ A_{2} $,
\item $ \mathscr{D} $ is a continuous linear isomorphism (not necessarily unitary) from $ \X_{1} $ to $ \X_{2} $,
\item $ \L_{2} \circ \mathscr{L} = \L_{1} $,
\item $ \func{\mathscr{D}}{a \bullet \zeta} = \func{\mathscr{L}}{a} \bullet \func{\mathscr{D}}{\zeta} $ for all $ a \in A_{1} $ and $ \zeta \in \X_{1} $,
\item $ \D_{2} \circ \mathscr{D} = \D_{1} $,
\item $ \Inner{\func{\mathscr{D}}{\cdot}}{\func{\mathscr{D}}{\cdot}}_{2} = \mathscr{L} \circ \Inn_{1} $.
\end{itemize}
\end{Def}

The following theorem says that the modular Gromov--Hausdorff propinquity has the required properties to qualify as a (class) pseudometric.

\begin{Thm}[\cite{L4}] Let $ \Trip{F}{G}{H} $ be an admissible triple. For all $ \Trip{F}{G}{H} $-metrized quantum vector bundles $ \Omega_{1} $, $ \Omega_{2} $ and $ \Omega_{3} $, the following statements hold:
\begin{enumerate}\itemsep=0pt
\item[{\rm 1)}] $ \ModProp{F}{G}{H}{\Omega_{1}}{\Omega_{2}} = \ModProp{F}{G}{H}{\Omega_{2}}{\Omega_{1}} $,
\item[{\rm 2)}] $ \ModProp{F}{G}{H}{\Omega_{1}}{\Omega_{3}} \leq \ModProp{F}{G}{H}{\Omega_{1}}{\Omega_{2}} + \ModProp{F}{G}{H}{\Omega_{2}}{\Omega_{3}}$,
\item[{\rm 3)}] $ \ModProp{F}{G}{H}{\Omega_{1}}{\Omega_{2}} = 0 $ if and only if there exists a full quantum isometry from $\Omega_{1}$ to $\Omega_{2}$.
\end{enumerate}
\end{Thm}

\begin{Lem} \label{The Effect of Scaling Riemannian Metrics on D-Norms} Let $ g $ be a Riemannian metric, and let $ r,s \in \R_{> 0} $. Then
\begin{gather*}
\Dom{\DM{r \cdot g}} = \Dom{\DM{s \cdot g}}
\qquad \text{and} \qquad
\frac{1}{\sqrt{r}} \DM{r \cdot g} = \frac{1}{\sqrt{s}} \DM{s \cdot g}.
\end{gather*}
\end{Lem}

\begin{proof}One can verify that $ \nabla^{r \cdot g} $ is also a Levi-Civita connection for $ g $, so $ \nabla^{r \cdot g} = \nabla^{g} $ by Theorem~\ref{The Fundamental Theorem of Riemannian Geometry for Generically Transcendental Quantum Tori} and
\begin{gather*}
\forall\, \delta \in \Der{\Theta}, ~ \forall\, X,Y \in \Smooth{\chi_{\Theta}}\colon \quad
 \Inner{\func{\nabla^{r \cdot g}_{\delta}}{X}}{Y}_{r \cdot g}
= \Inner{\func{\nabla^{g}_{\delta}}{X}}{Y}_{r \cdot g}
= r \cdot \Inner{\func{\nabla^{g}_{\delta}}{X}}{Y}_{g},
\end{gather*}
which yields
\begin{gather*}
\forall\, \delta \in \Der{\Theta}, ~ \forall\, X,Y \in \Smooth{\chi_{\Theta}}\colon \quad
 \big\|\Inner{\func{\nabla^{r \cdot g}_{\delta}}{X}}{Y}_{r \cdot g}\big\|_{A_{\Theta}}
= \sqrt{r} \big\|\Inner{\func{\nabla^{g}_{\delta}}{X}}{\sqrt{r} \cdot Y}_{g}\big\|_{A_{\Theta}}.
\end{gather*}
Hence, we have for all $ \delta \in \Der{\Theta} $ and $ X \in \Smooth{\chi_{\Theta}} $ that
\begin{align*}
 \Norm{\func{\nabla^{r \cdot g}_{\delta}}{X}}_{r \cdot g}
& = \sup_{\substack{Y \in \chi_{\Theta} \\ \Norm{Y}_{r \cdot g} = 1}}
 \big\|\Inner{\func{\nabla^{r \cdot g}_{\delta}}{X}}{Y}_{r \cdot g}\big\|_{A_{\Theta}}
 = \sup_{\substack{Y \in \chi_{\Theta} \\ \Norm{Y}_{r \cdot g} = 1}}
 \sqrt{r} \big\|\Inner{\func{\nabla^{g}_{\delta}}{X}}{\sqrt{r} \cdot Y}_{g}\big\|_{A_{\Theta}} \\
& = \sup_{\substack{Y \in \chi_{\Theta} \\ \Norm{\sqrt{r} \cdot Y}_{g} = 1}}
 \sqrt{r} \big\|\Inner{\func{\nabla^{g}_{\delta}}{X}}{\sqrt{r} \cdot Y}_{g}\big\|_{A_{\Theta}} \quad
 \big(\text{as $ \Norm{\cdot}_{r \cdot g} = \sqrt{r} \Norm{\cdot}_{g} $}\big) \\
& = \sup_{\substack{Y \in \chi_{\Theta} \\ \Norm{Y}_{g} = 1}} \sqrt{r} \big\|\Inner{\func{\nabla^{g}_{\delta}}{X}}{Y}_{g}\big\|_{A_{\Theta}}
 = \sqrt{r} \Norm{\func{\nabla^{g}_{\delta}}{X}}_{g}.
\end{align*}
It follows readily that
\begin{gather*}
\forall\, X \in \Smooth{\chi_{\Theta}}\colon \quad
 \OpNorm{X}_{r \cdot g}
= \sup_{\substack{\delta \in \Der{\Theta} \\ \Norm{\delta}_{\Der{\Theta}} \leq 1}} \Norm{\func{\nabla^{r \cdot g}_{\delta}}{X}}_{r \cdot g}
= \sqrt{r} \sup_{\substack{\delta \in \Der{\Theta} \\ \Norm{\delta}_{\Der{\Theta}} \leq 1}} \Norm{\func{\nabla^{g}_{\delta}}{X}}_{g}
= \sqrt{r} \OpNorm{X}_{g}.
\end{gather*}
We thus obtain $ \OpNorm{\cdot}_{r \cdot g} = \sqrt{r} \OpNorm{\cdot}_{g} $, so
\begin{gather*}
 \D_{r \cdot g}
= \max \big(\Norm{\cdot}_{r \cdot g},\OpNorm{\cdot}_{r \cdot g}\big)
= \max \big(\sqrt{r} \Norm{\cdot}_{g},\sqrt{r} \OpNorm{\cdot}_{g}\big)
= \sqrt{r} \max\big(\Norm{\cdot}_{g},\OpNorm{\cdot}_{g}\big)
= \sqrt{r} \D_{g}.
\end{gather*}
Consequently, $ \D_{g}^{1} = \D_{r \cdot g}^{\sqrt{r}} = \sqrt{r} \cdot \D_{r \cdot g}^{1} $.

Now, we have by Definition~\ref{Minkowski Gauge Functional} that
\begin{align*}
 \Dom{\DM{g}}
&  = \func{\sup}{
 \left\{{X \in \chi_{\Theta}}\,\Big|\,
 {\text{there exists a} ~ t \in \R_{> 0} ~ \text{such that} ~ X \in t \cdot \Cl{\D_{g}^{1}}{\Norm{\cdot}_{g}}}\right\}
 } \\
&   = \func{\sup}{
 \left\{{X \in \chi_{\Theta}}\,\Big|\,
 {
 \text{there exists a} ~ t \in \R_{> 0} ~ \text{such that} ~
 X \in t \cdot \Cl{\sqrt{r} \cdot \D_{r \cdot g}^{1}}{\Norm{\cdot}_{g}}
 }\right\}
 } \\
&   = \func{\sup}{
 \left\{{X \in \chi_{\Theta}}\,\Big|\,
 {
 \text{there exists a} ~ t \in \R_{> 0} ~ \text{such that} ~
 X \in t \sqrt{r} \cdot \Cl{\D_{r \cdot g}^{1}}{\Norm{\cdot}_{g}}
 }\right\}
 } \\
& = \func{\sup}{
 \left\{{X \in \chi_{\Theta}}\,\Big|\,
 {
 \text{there exists a} ~ t \in \R_{> 0} ~ \text{such that} ~
 X \in t \sqrt{r} \cdot \Cl{\D_{r \cdot g}^{1}}{\Norm{\cdot}_{r \cdot g}}
 }\right\}
 } \\
 & \hphantom{=}{} \ \big(\text{as $ \Norm{\cdot}_{g} $ and $ \Norm{\cdot}_{r \cdot g} $ are equivalent}\big) \\
&  = \func{\sup}{
 \left\{{X \in \chi_{\Theta}}\,\Big|\,
 {
 \text{there exists a} ~ t \in \R_{> 0} ~ \text{such that} ~
 X \in t \cdot \Cl{\D_{r \cdot g}^{1}}{\Norm{\cdot}_{r \cdot g}}
 }\right\}
 } \\
&  = \Dom{\DM{r \cdot g}}.
\end{align*}
Next, observe for all $ X \in \Dom{\DM{g}} = \Dom{\DM{r \cdot g}} $ that
\begin{align*}
 \DMArg{g}{X}
& = \func{\inf}{\left\{{t \in \R_{> 0}}\,\Big|\,{X \in t \cdot \Cl{\D_{g}^{1}}{\Norm{\cdot}_{g}}}\right\}} \\
& = \func{\inf}{\left\{{t \in \R_{> 0}}\,\Big|\,{X \in t \cdot \Cl{\sqrt{r} \cdot \D_{r \cdot g}^{1}}{\Norm{\cdot}_{g}}}\right\}} \\
& = \func{\inf}{\left\{{t \in \R_{> 0}}\,\Big|\,{X \in t \sqrt{r} \cdot \Cl{\D_{r \cdot g}^{1}}{\Norm{\cdot}_{g}}}\right\}} \\
& = \func{\inf}{\left\{{t \in \R_{> 0}}\,\Big|\,{X \in t \sqrt{r} \cdot \Cl{\D_{r \cdot g}^{1}}{\Norm{\cdot}_{r \cdot g}}}\right\}} \quad
 \big(\text{as $ \Norm{\cdot}_{g} $ and $ \Norm{\cdot}_{r \cdot g} $ are equivalent}\big) \\
& = \frac{1}{\sqrt{r}} \func{\inf}{\left\{{t \in \R_{> 0}}\,\Big|\,{X \in t \cdot \Cl{\D_{r \cdot g}^{1}}{\Norm{\cdot}_{r \cdot g}}}\right\}} \\
& = \frac{1}{\sqrt{r}} \DMArg{r \cdot g}{X}.
\end{align*}
As $ X \in \Dom{\DM{g}} = \Dom{\DM{r \cdot g}} $ is arbitrary, we obtain $ \DM{r \cdot g} = \sqrt{r} \DM{g} $.

Similarly, $ \Dom{\DM{g}} = \Dom{\DM{s \cdot g}} $ and $ \DM{s \cdot g} = \sqrt{s} \DM{g} $, so
\begin{gather*}
\Dom{\DM{r \cdot g}} = \Dom{\DM{s \cdot g}}
\qquad \text{and} \qquad
\dfrac{1}{\sqrt{r}} \DM{r \cdot g} = \dfrac{1}{\sqrt{s}} \DM{s \cdot g}
\end{gather*}
as required.
\end{proof}

We can now verify without difficulty that $ \Pair{\mathscr{L}}{\mathscr{D}} $ is a full quantum isometry from $ \Omega_{r} $ to $ \Omega_{s} $, where
\begin{gather*}
\mathscr{L} = \Id_{A_{\Theta} \to A_{\Theta}} \qquad \text{and} \qquad \mathscr{D} = \sqrt{\frac{r}{s}} \Id_{\chi_{\Theta} \to \chi_{\Theta}}.
\end{gather*}
However, in order to give the reader an appreciation for the sophistication involved when \smash{trying} to resolve more difficult propinquity problems (see the conclusion below), we will prove Theorem~\ref{Second Main Theorem} directly.

\begin{proof}[Proof of Theorem~\ref{Second Main Theorem}]
By Lemma~\ref{The Effect of Scaling Riemannian Metrics on D-Norms}, we may define a bijection $ \beta_{r,s}\colon \DM{r \cdot g}^{1} \to \DM{s \cdot g}^{1} $ by
\begin{gather*}
\forall\, X \in \DM{r \cdot g}^{1}\colon \quad
\func{\beta_{r,s}}{X} \df \sqrt{\frac{r}{s}} \cdot X.
\end{gather*}
It follows that
\begin{gather*}
 \gamma_{r,s}
\df \big( \Omega_{r},\Omega_{s},A_{\Theta},1_{A_{\Theta}},\Id_{A_{\Theta} \to A_{\Theta}},\Id_{A_{\Theta} \to A_{\Theta}},
 \DM{r \cdot g}^{1},\Id_{\DM{r \cdot g}^{1} \to \DM{r \cdot g}^{1}},\beta_{r,s}\big)
\end{gather*}
is an $ \Trip{F}{G}{H} $-modular bridge, whose associated numerical quantities we now seek to compute.

To perform these computations, first note that $ \bn_{\gamma_{r,s}} $ is a seminorm on $ A_{\Theta} \oplus A_{\Theta} $ such that
\begin{align*}
\forall\, a_{1},a_{2} \in A_{\Theta}\colon \quad
 \bnArg{\gamma_{r,s}}{a_{1}}{a_{2}}& = \Norm{\func{\Id_{A_{\Theta} \to A_{\Theta}}}{a_{1}} 1_{A_{\Theta}} - 1_{A_{\Theta}} \func{\Id_{A_{\Theta} \to A_{\Theta}}}{a_{2}}} _{A_{\Theta}}\\
 & = \Norm{a_{1} - a_{2}}_{A_{\Theta}}.
\end{align*}
We can then make the following observations regarding the basic reach, height and imprint of~$ \gamma_{r,s} $:
\begin{itemize}\itemsep=0pt
\item The basic reach of $ \gamma_{r,s} $ is $ 0 $:
\begin{align*}
 \func{\rho_{\flat}}{\gamma_{r,s}}
& = \func{\max}{
 \adjustlimits \sup_{\substack{a_{1} \in \Dom{\L} \\ \func{\L}{a_{1}} \leq 1}}
 \inf_{\substack{a_{2} \in \Dom{\L} \\ \func{\L}{a_{2}} \leq 1}}
 \Norm{a_{1} - a_{2}}_{A_{\Theta}},
 \adjustlimits \sup_{\substack{a_{2} \in \Dom{\L} \\ \func{\L}{a_{2}} \leq 1}}
 \inf_{\substack{a_{1} \in \Dom{\L} \\ \func{\L}{a_{1}} \leq 1}}
 \Norm{a_{1} - a_{2}}_{A_{\Theta}}
 } \\
& = \func{\max}{
 \sup_{\substack{a_{1} \in \Dom{\L} \\ \func{\L}{a_{1}} \leq 1}} 0,
 \sup_{\substack{a_{2} \in \Dom{\L} \\ \func{\L}{a_{2}} \leq 1}} 0
 } = \func{\max}{0,0} = 0.
\end{align*}

\item The height of $ \gamma_{r,s} $ is $ 0 $: As $ \State{A_{\Theta}|1_{A_{\Theta}}} = \State{A_{\Theta}} $, we have
\begin{align*}
 \func{\varsigma}{\gamma_{r,s}}
& = \func{\max}{
 \begin{matrix}
 \Haus{\mk_{\L}}{\State{A_{\Theta}}}{\Set{\phi \circ \Id_{A_{\Theta} \to A_{\Theta}}}{\phi \in \State{A_{\Theta}}}} \\
 \Haus{\mk_{\L}}{\State{A_{\Theta}}}{\Set{\psi \circ \Id_{A_{\Theta} \to A_{\Theta}}}{\psi \in \State{A_{\Theta}}}}
 \end{matrix}
 } \\
& = \max\big(\Haus{\mk_{\L}}{\State{A_{\Theta}}}{\State{A_{\Theta}}},\Haus{\mk_{\L}}{\State{A_{\Theta}}}{\State{A_{\Theta}}}\big) \\
& = \func{\max}{0,0} = 0.
\end{align*}

\item The imprint of $ \gamma_{r,s} $ is $ 0 $:
\begin{align*}
 \func{\varpi}{\gamma_{r,s}}
& = \max\big(
\mathsf{Haus}_{\k_{\Omega_{r}}}\big(\operatorname{Range}\big(\Id_{\DM{r \cdot g}^{1} \to \DM{r \cdot g}^{1}}\big),\DM{r \cdot g}^{1}\big),\\
& \hphantom{=\max\big(}{}
 \Haus{\k_{\Omega_{s}}}{\Range{\beta_{r,s}}}{\DM{s \cdot g}^{1}}
 \big) \\
& = \max\big(
 \Haus{\k_{\Omega_{r}}}{\DM{r \cdot g}^{1}}{\DM{r \cdot g}^{1}},
 \Haus{\k_{\Omega_{s}}}{\DM{s \cdot g}^{1}}{\DM{s \cdot g}^{1}}
 \big) \\
& = \func{\max}{0,0} = 0.
\end{align*}
\end{itemize}
Hence, $ \func{\lambda}{\gamma_{r,s}} = \func{\rho^{\sharp}}{\gamma_{r,s}} $, i.e., the length of $ \gamma_{r,s} $ equals its modular reach, which we will now prove is also $ 0 $.

Observe for all $ X,Y \in \DM{r \cdot g}^{1} $ that
\begin{align*}
\mathsf{bn}_{\gamma_{r,s}}\big(\Inner{X}{Y}_{r \cdot g},\Inner{\func{\beta_{r,s}}{X}}{\func{\beta_{r,s}}{Y}}_{s \cdot g}\big)
& = \big\|\Inner{X}{Y}_{r \cdot g} - \Inner{\func{\beta_{r,s}}{X}}{\func{\beta_{r,s}}{Y}}_{s \cdot g}\big\|_{A_{\Theta}} \\
& = \Norm{\Inner{X}{Y}_{r \cdot g} - \Inner{\sqrt{\frac{r}{s}} \cdot X}{\sqrt{\frac{r}{s}} \cdot Y}_{s \cdot g}}_{A_{\Theta}} \\
& = \Norm{\Inner{X}{Y}_{r \cdot g} - \frac{r}{s} \cdot \Inner{X}{Y}_{s \cdot g}}_{A_{\Theta}} \\
& = \Norm{r \cdot \Inner{X}{Y}_{g} - \frac{r}{s} \cdot \big(s \cdot \Inner{X}{Y}_{g}\big)}_{A_{\Theta}} \\
& = \big\|r \cdot \Inner{X}{Y}_{g} - r \cdot \Inner{X}{Y}_{g}\big\|_{A_{\Theta}} = 0.
\end{align*}
As the argument is symmetric in $ X $ and $ Y $, we also have for all $ X,Y \in \DM{r \cdot g}^{1} $ that
\begin{gather*}
\mathsf{bn}_{\gamma_{r,s}}\big(\Inner{Y}{X}_{r \cdot g},\Inner{\func{\beta_{r,s}}{Y}}{\func{\beta_{r,s}}{X}}_{s \cdot g}\big) = 0,
\end{gather*}
so
\begin{gather*}
\mathsf{dn}_{\gamma_{r,s}}\big(\func{\Id_{\DM{r \cdot g}^{1} \to \DM{r \cdot g}^{1}}}{X},\func{\beta_{r,s}}{X}\big) = \dnArg{\gamma_{r,s}}{X}{\func{\beta_{r,s}}{X}} \\
\qquad = \func{\sup}{
 \Set{
 \begin{matrix}
\mathsf{bn}_{\gamma_{r,s}}
 \big(\big\langle X\,|\, \func{\Id_{\DM{r \cdot g}^{1} \to \DM{r \cdot g}^{1}}}{Y}\big\rangle_{r \cdot g},
 \Inner{\func{\beta_{r,s}}{X}}{\func{\beta_{r,s}}{Y}}_{s \cdot g}\big) \\
\mathsf{bn}_{\gamma_{r,s}}
\big(\big\langle\func{\Id_{\DM{r \cdot g}^{1} \to \DM{r \cdot g}^{1}}}{Y}\,|\,X\big\rangle_{r \cdot g},
 \Inner{\func{\beta_{r,s}}{Y}}{\func{\beta_{r,s}}{X}}_{s \cdot g}\big)
 \end{matrix}
 }
 {Y \in \DM{r \cdot g}^{1}}
 } \\
\qquad{} = \func{\sup}{
 \Set{
 \begin{matrix}
\mathsf{bn}_{\gamma_{r,s}}\big(\Inner{X}{Y}_{r \cdot g},\Inner{\func{\beta_{r,s}}{X}}{\func{\beta_{r,s}}{Y}}_{s \cdot g}\big) \\
\mathsf{bn}_{\gamma_{r,s}}\big(\Inner{Y}{X}_{r \cdot g},\Inner{\func{\beta_{r,s}}{Y}}{\func{\beta_{r,s}}{X}}_{s \cdot g}\big)
 \end{matrix}
 }
 {Y \in \DM{r \cdot g}^{1}}
 } = 0
\end{gather*}
for all $ X \in \DM{r \cdot g}^{1} $. Therefore,
\begin{align*}
 \ModProp{F}{G}{H}{\Omega_{r}}{\Omega_{s}}
& \leq \func{\lambda}{\gamma_{r,s}} = \func{\rho^{\sharp}}{\gamma_{r,s}} \\
& = \sup\big(
 \big\{ \mathsf{dn}_{\gamma_{r,s}}\big(\func{\Id_{\DM{r \cdot g}^{1} \to \DM{r \cdot g}^{1}}}{X},\func{\beta_{r,s}}{X}\big) \,|\, {X \in \DM{r \cdot g}^{1}} \big\}\big) = 0.
\end{align*}
This completes the proof.
\end{proof}

\section{Conclusion}\label{section6}

The only property of Levi-Civita connections used in this paper is metric compatibility. We could have allowed our $ \D $ norms to depend not only on the choice of a Riemannian metric, but also on the choice of a connection that is not necessarily torsion-free. However, this would introduce an extra degree of variability that could unnecessarily complicate the study of convergence questions in modular propinquity.

Many unanswered questions remain. The four most important ones (in our opinion) are the following:
\begin{itemize}\itemsep=0pt
\item Is $ \D_{g} $ lower-semicontinuous on $ \Smooth{\chi_{\Theta}} $ for any Riemannian metric $ g $? If this is the case, then (4) of Lemma~\ref{Some Facts About Minkowski Gauge Functionals} says that $ \DM{g} $ extends $ \D_{g} $, which would render $ \DM{g} $ more manageable to deal with.
\item Can we find a topology on $ \func{M_{n}}{A_{\Theta}} $ so that if $ h \to g $ in the space of Riemannian metrics with respect to this topology, then
\begin{gather*}
\ModProp{F}{G}{H}{\big(\chi_{\Theta},\Inn_{g},\DM{g},A_{\Theta},\L\big)}{\big(\chi_{\Theta},\Inn_{h},\DM{h},A_{\Theta},\L\big)} \to 0?
\end{gather*}
This question can be posed so straightforwardly because we consider only Levi-Civita connections.

\item Can we incorporate the Riemannian curvature operator \cite[Definition~3.1]{Ro} into the current work?

\item Although Rosenberg considers only free modules in \cite{Ro}, can we extend our results to non-free modules? In~\cite{L5}, Latr\'emoli\`ere establishes a technical propinquity result for metrized quantum vector bundles whose underlying modules are non-free Heisenberg modules over quantum $ 2 $-tori, but it is not clear at this time how his result may be extended to non-free Hilbert modules over quantum tori of arbitrary dimension.
\end{itemize}

\subsection*{Acknowledgments}

The author wishes to thank Konrad Aguilar and Fr\'ed\'eric Latr\'emoli\`ere for their generous help on the Gromov--Hausdorff propinquity. The author also wishes to thank Albert Sheu for his help in understanding Rosenberg's work on Levi-Civita connections. The warmest gratitude is reserved for the referees who offered valuable advice on how to significantly improve the quality of this paper.

\pdfbookmark[1]{References}{ref}
\LastPageEnding

\end{document}